\documentclass[11pt]{article}

\usepackage{amsmath}
\usepackage{amscd}
\usepackage{amssymb}
\usepackage{amsthm}
\usepackage{enumitem}
\usepackage{color}
\usepackage{tikz-cd}
\usepackage{hyperref}
\usepackage{cleveref}
\usepackage{mathtools}
\usepackage{caption}
\usepackage{todonotes}
\usepackage{cancel}
\usepackage[margin=30mm,right=30mm, top=30mm]{geometry}

\begin{document}
\input{epsf}

\sloppy

\vspace{3cm}
\title{Construction of elliptic $\frp$-units}

\author{W. Bley and M. Hofer}



\makeatletter
\newtheorem*{rep@theorem}{\rep@title}
\newcommand{\newreptheorem}[2]{%
\newenvironment{rep#1}[1]{%
 \def\rep@title{#2 \ref{##1}}%
 \begin{rep@theorem}}%
 {\end{rep@theorem}}}
\makeatother

\newtheorem{Theorem}{Theorem}[section]
\newtheorem{Lemma}[Theorem]{Lemma}
\newtheorem{Corollary}[Theorem]{Corollary}
\newtheorem{Definition}[Theorem]{Definition}
\newtheorem{Remark}[Theorem]{Remark}
\newtheorem{Remarks}[Theorem]{Remarks}
\newtheorem{Proposition}[Theorem]{Proposition}
\newtheorem{Conjecture}[Theorem]{Conjecture}
\newtheorem{Example}[Theorem]{Example}
\newtheorem{Algorithm}[Theorem]{Algorithm}
\newtheorem{Examples}[Theorem]{Examples}

\newtheorem{Assumption}[Theorem]{Assumption}
\newtheorem{Deflemma}[Theorem]{Definition/Lemma}
\newtheorem{Claim}{Claim}
\newtheorem*{claim_wn}{Claim}
\newtheorem*{Hypotheses}{Hypotheses}
\newtheorem*{theorem_wn}{Main Theorem}

\newreptheorem{theorem}{Theorem}

\newenvironment{theorem}{\begin{Theorem}\sl}{\end{Theorem}}
\newenvironment{corollary}{\begin{Corollary}\sl}{\end{Corollary}}          %
\newenvironment{remark}{\begin{Remark}\rm}{\end{Remark}}
\newenvironment{remarks}{\begin{Remarks}\rm}{\end{Remarks}}
\newenvironment{definition}{\begin{Definition}\rm}{\end{Definition}}
\newenvironment{lemma}{\begin{Lemma}\sl}{\end{Lemma}}
\newenvironment{proposition}{\begin{Proposition}\sl}{\end{Proposition}}
\newenvironment{conjecture}{\begin{Conjecture}\sl}{\end{Conjecture}}
\newenvironment{example}{\begin{Example}}{\end{Example}}
\newenvironment{algorithm}{\begin{Algorithm}}{\end{Algorithm}}
\newenvironment{examples}{\begin{Examples}\rm}{\end{Examples}}
\newenvironment{deflemma}{\begin{Deflemma}\rm}{\end{Deflemma}}
\newenvironment{assumption}{\begin{Assumption}\rm}{\end{Assumption}}
\newenvironment{claim}{\begin{Claim}\rm}{\end{Claim}}
\newenvironment{hypotheses}{\begin{Hypotheses}\rm}{\end{Hypotheses}}

\newcommand{\A}{\mathbf A}
\newcommand{\ahat}{\hat{a}}
\newcommand{\B}{\mathbf B}
\newcommand{\C}{\mathbf C}
\newcommand{\calA}{\mathcal{A}}
\newcommand{\calB}{\mathcal{B}}
\newcommand{\calC}{\mathcal{C}}
\newcommand{\calD}{\mathcal{D}}
\newcommand{\calE}{\mathcal{E}}
\newcommand{\calF}{\mathcal{F}}
\newcommand{\calG}{\mathcal{G}}
\newcommand{\calH}{\mathcal{H}}
\newcommand{\calJ}{\mathcal{J}}
\newcommand{\calK}{\mathcal{K}}
\newcommand{\calL}{\mathcal{L}}
\newcommand{\calM}{\mathcal{M}}
\newcommand{\calN}{\mathcal{N}}
\newcommand{\calP}{\mathcal{P}}
\newcommand{\calO}{\mathcal{O}}
\newcommand{\calR}{\mathcal{R}}
\newcommand{\calS}{\mathcal{S}}
\newcommand{\calU}{\mathcal{U}}
\newcommand{\calV}{\mathcal{V}}
\newcommand{\calX}{\mathcal{X}}
\newcommand{\End}{\mathrm{End}}
\newcommand{\Fbar}{\bar{F}}
\newcommand{\frA}{\mathfrak{A}}
\newcommand{\fra}{\mathfrak{a}}
\newcommand{\frb}{\mathfrak{b}}
\newcommand{\frc}{\mathfrak{c}}
\newcommand{\frd}{\mathfrak{d}}
\newcommand{\frf}{\mathfrak{f}}
\newcommand{\frg}{\mathfrak{g}}
\newcommand{\frl}{\mathfrak{l}}
\newcommand{\frm}{\mathfrak{m}}
\newcommand{\frn}{\mathfrak{n}}
\newcommand{\frt}{\mathfrak{t}}
\newcommand{\frP}{\mathfrak{P}}
\newcommand{\frp}{\mathfrak{p}}
\newcommand{\frQ}{\mathfrak{Q}}
\newcommand{\frq}{\mathfrak{q}}
\newcommand{\frw}{\mathfrak{w}}
\newcommand{\frJ}{\mathfrak{J}}
\newcommand{\frW}{\mathfrak{W}}
\newcommand{\Gal}{\mathrm{Gal}}
\newcommand{\Ghat}{\hat{G}}
\newcommand{\Hom}{\mathrm{Hom}}
\newcommand{\Aut}{\mathrm{Aut}}
\newcommand{\Kbar}{{\bar{K}}}
\newcommand{\Lbar}{{\bar{L}}}
\newcommand{\Mbar}{{\bar{M}}}
\newcommand{\Ltil}{{\tilde{L}}}
\newcommand{\lexp}[2]{\setbox0=\hbox{$#2$} \setbox1=\vbox to
         \ht0{}\,\box1^{#1}\!#2}
\newcommand{\Map}{\mathrm{Map}}
\newcommand{\Na}{\mathbb{N}}
\newcommand{\Om}{\Omega}
\newcommand{\OmL}{\Omega_L}
\newcommand{\om}{\omega}
\newcommand{\pringkp}{\calO_{k'}[[X]]}
\newcommand{\pip}{\pi^\prime}
\newcommand{\pipm}{{\pi^{\prime (m)}}}
\newcommand{\Ags}{\mathcal{A}^{\mathrm{gs}}}
\newcommand{\ALgs}{\mathcal{A}_L^{\mathrm{gs}}}
\newcommand{\Bgs}{\mathcal{B}^{\mathrm{gs}}}
\newcommand{\BLgs}{\mathcal{B}_L^{\mathrm{gs}}}
\newcommand{\TrBK}{\mathrm{Tr}_{B/K}}
\newcommand{\ev}{\mathrm{ev}}
\newcommand{\gdw}{\mathop{\Longleftrightarrow}}
\newcommand{\Qu}{\mathbb{Q}}
\newcommand{\Qbar}{\bar{\Qu}}
\newcommand{\stab}{\mathrm{stab}}
\newcommand{\res}{\mathrm{res}}
\newcommand{\N}{\calN}
\newcommand{\cl}{\mathrm{cl}}
\newcommand{\Tr}{\mathrm{Tr}}
\newcommand{\Ze}{\mathbb{Z}}
\newcommand{\Ce}{\mathbb{C}}
\renewcommand{\C}{\mathbb{C}}
\renewcommand{\Re}{\mathbb{R}}
\renewcommand{\O}{ {\calO} }
\newcommand{\Ok}{{\calO_k}}
\newcommand{\OF}{{\calO_F}}
\newcommand{\OK}{\calO_k}
\newcommand{\OFp}{{\calO_{\tilde F}}}
\newcommand{\OL}{{\calO_L}}
\newcommand{\OLw}{{\calO_{L_w}}}
\renewcommand{\OE}{{\calO_E}}
\newcommand{\OLp}{{\calO_{L'}}}
\newcommand{\OM}{{\calO_M}}
\newcommand{\ON}{{\calO_N}}
\newcommand{\OP}{{\calO_{P}}}
\newcommand{\Okfrp}{{\calO_{k_\frp}}}
\newcommand{\OPtilde}{{\tilde\calO_{P}}}
\newcommand{\Op}{{\calO'}}
\newcommand{\OpX}{{\calO'[[X]]}}
\newcommand{\mpar}{\marginpar}
\newcommand{\Etor}{{E_{\mathrm{tor}}}}
\newcommand{\lcm}{\mathrm{lcm}}
\newcommand{\Sl}{ {\mathrm{Sl}_2(\Ze)} }
\newcommand{\Arf}[1]{\stackrel{#1}{\longrightarrow}}
\newcommand{\lto}{\longmapsto}
\newcommand{\zpr}{ \zeta_{p^r} }
\newcommand{\zpri}{ \zeta_{p_i^{r_i}} }
\newcommand{\zpm}{ \zeta_{p^m} }
\newcommand{\zprm}{ \zeta_{p^{r+m}} }
\newcommand{\lat}{\frw}
\newcommand{\lathat}{{\hat{\lat}}}
\newcommand{\Lhat}{ {\hat{L} }}
\renewcommand{\mod}{\mathrm{mod\ }}
\newcommand{\chiell}{\chi_{\mathrm{ell}}}
\gdef\star#1{ \,^*#1 }

\newcommand{\R}{ \mathbb{R} }
\newcommand{\Q}{ \mathbb{Q} }
\renewcommand{\Im}{ \mathrm{Im} }

\gdef \fctb#1#2#3#4{
     #1 \left( #2 \left\vert {#3 \atop #4} \right. \right) }

\gdef \fcta#1#2#3{
     #1 \left( #2 \left\vert {#3} \right. \right) }

\gdef \th#1{ \theta_0( {#1}, \fra ) }
\gdef \thJ{ \theta_0( {\tau_J}, \fra ) }
\gdef \sgal#1#2{ \sum_{g \in \Gal(k(#1)/k(#2))} }
\gdef \sgalk#1{ \sum_{g \in \Gal(k(#1)/k)} }

\newcommand{\charpol}{ {\mathrm{char}} }
\newcommand{\CG}{ {\Ce [G]} }
\newcommand{\ZG}{ {\Ze [G]} }
\newcommand{\ZGv}{ {\Ze [G_v]} }
\newcommand{\QGv}{ {\Qu [G_v]} }
\newcommand{\CHa}{ {\cal{C}_\mathrm{Ha}} }
\newcommand{\Cl}{ {\Ce_l} }
\newcommand{\Cp}{ {\Ce_p} }
\newcommand{\CN}{ {\calC_N} }
\newcommand{\DS}{ {\Delta S} }
\newcommand{\diff}{\frd_k}
\newcommand{\im}{\text{im}}
\newcommand{\ind}{\text{ind}}
\newcommand{\infl}{\text{infl}}
\newcommand{\indGp}{\text{ind}_{G_\frP}^G}
\newcommand{\lOF}{\ell_\OF}
\newcommand{\lmt}{\longmapsto}
\newcommand{\lra}{\longrightarrow}
\newcommand{\la}{\leftarrow}
\newcommand{\nI}{ {\frn_I} }
\newcommand{\nJ}{ {\frn_J} }
\newcommand{\Oprime}{ {\calO'} }
\newcommand{\Ol}{{\calO_l}}
\newcommand{\OS}{{\calO_S}}
\newcommand{\OLS}{\calO_{L,S}}
\newcommand{\ULS}{U_{L,S}}
\newcommand{\ptJ}{\varphi^{12N(\frb\nJ)}(\tau_J \mid L)}
\newcommand{\QG}{ {\Qu [G]} }
\newcommand{\QlG}{ {\Qu_l [G]} }
\newcommand{\RG}{ {\Re [G]} }
\newcommand{\Ql}{ {\Qu_l} }
\newcommand{\Qp}{ {\Qu_p} }
\newcommand{\QpG}{ {\Qu_p[G]} }
\newcommand{\Zp}{ {\Ze_p} }
\newcommand{\ZpG}{ {\Ze_p[G]} }
\newcommand{\ra}{\rightarrow}
\newcommand{\Sinf}{S_\infty}
\newcommand{\SRam}{ S^\mathrm{ram} }
\newcommand{\Sram}{ S_\mathrm{ram} }
\newcommand{\tensor}{ {\otimes} }
\newcommand{\tensorZ}{ {\otimes_\Ze} }
\newcommand{\tensorQ}{ {\otimes_\Qu} }
\newcommand{\ZS}{\Ze S}
\newcommand{\Zl}{ {\Ze_l} }
\newcommand{\ZlG}{ {\Ze_l[G]} }
\newcommand{\ZD}{ {\Ze[D]} }
\newcommand{\QD}{ {\Qu[D]} }
\newcommand{\uphi}{\underline{\varphi}}
\newcommand{\Vcheck}{ \check{V} }
\newcommand{\TO}{ T\Omega(L/K) }
\newcommand{\TpO}{ T_p\Omega(L/K) }
\newcommand{\TOa}{ T\Omega_\alpha }
\newcommand{\Ext}{ \mathrm{Ext} }
\newcommand{\bfD}{ \bf{D} }
\newcommand{\D}{ \calD }
\newcommand{\IsEG}{ \mathrm{Is}_{E [\Gamma]} }
\newcommand{\IsRG}{ \mathrm{Is}_{\Re [G]} }
\newcommand{\PsiSR}{ \Psi_{S, \Re} }
\newcommand{\XSR}{ X_{S, \Re} }
\newcommand{\USR}{ U_{S, \Re} }
\newcommand{\Fitt}{ \mathrm{Fitt} }
\renewcommand{\det}{ \mathrm{det} }
\newcommand{\cok}{ \mathrm{cok} }
\newcommand{\Det}{ \mathrm{Det} }
\newcommand{\Spec}{ \mathrm{Spec} }
\newcommand{\ext}{ \mathrm{Ext}^2_G(X_S, U_S) }
\newcommand{\pext}{ \mathrm{CI}^2_G(X_S, U_S) }
\newcommand{\hext}{ \mathrm{Hom}_G(\Sigma,U_S) }
\newcommand{\hextn}{ \mathrm{Hom}_G(\Sigma,N) }
\newcommand{\ppext}{ \mathrm{CI}_G(\Sigma,U_S) }
\newcommand{\extn}{ \mathrm{Ext}^2_G(X_S, N) }
\newcommand{\ppextn}{ \mathrm{CI}_G(\Sigma,N) }
\newcommand{\extbar}{ \mathrm{Ext}^2_G(X_S,\bar{U}_S) }
\newcommand{\pextbar}{ \mathrm{CI}^2_G(X_S,\bar{U}_S) }
\newcommand{\ppextbar}{ \mathrm{CI}_G(\Sigma,\bar{U}) }
\newcommand{\Krel}{ K_0(\ZG, \Qu) }
\newcommand{\USbar}{ \bar{U} }
\newcommand{\XSmtwo}{ X_S(-2) }
\newcommand{\connect}{ \delta_{G} }
\newcommand{\connectl}{ \delta_{G,l} }
\newcommand{\Whe}{ \mathrm{Whe}_G }
\newcommand{\bWhe}{ \mathrm{\bf Whe}_G }
\newcommand{\bHom}{ \mathrm{\bf Hom} }
\newcommand{\Pic}{ \mathrm{Pic} }
\newcommand{\IsA}{ \mathrm{Is}_A }
\newcommand{\underscore}{ \ \_\   }
\newcommand{\pinfty}{ \frp_\infty }
\newcommand{\zla}{ \zeta_{l^a} }
\gdef \zl#1#2{ \zeta_{l_#1^#2} }
\newcommand{\vareps}{ \varepsilon }
\newcommand{\eps}{ \varepsilon }
\newcommand{\bfzero}{ {\bf 0} }
\newcommand{\ZpT}{ {\Zp[[T]]} }
\newcommand{\pd}{ {\mathrm{pd}} }
\newcommand{\Ipos}{ f_{\omega\psi^{-1}}(T) }
\newcommand{\ZlonexiG}{ {\Ze_{l}(\xi)[G_{l}]} }
\newcommand{\ord}{ \mathrm{ord} }
\newcommand{\Punkt}{ {\bf \cdot} }
\newcommand{\ETNC}{ETNC }
\newcommand{\Gl}{\mathrm{Gl}}
\newcommand{\Mat}{\mathrm{Mat}}
\newcommand{\US}{U_S}
\newcommand{\ES}{\calE_S}
\newcommand{\UKS}{U_{K,S}}
\newcommand{\rk}{ \mathrm{rk} }
\newcommand{\SL}{ {\Sigma(L)} }
\newcommand{\SK}{ {\Sigma(K)} }
\newcommand{\piLK}{ {\pi_{L/K} } }
\newcommand{\xiLK}{ {\xi_{L/K} } }
\newcommand{\tauLK}{ {\tau_{L/K} } }
\newcommand{\HL}{ {H_L} }
\newcommand{\HLZ}{ {H_{L, \Ze}}}
\newcommand{\HLK}{ {H_{L/K}}}
\newcommand{\HLKZ}{ {H_{L/K, \Ze}}}
\newcommand{\WLK}{ {{\cal W}(L/K)} }
\newcommand{\OKG}{ {{\cal O}_K [G]} }
\newcommand{\OKvG}{ {{\cal O}_{K_v} [G]} }
\newcommand{\OKvGv}{ {{\cal O}_{K_v} [G_v]} }
\newcommand{\Gm}{ {{\bf G}_m} }
\newcommand{\KwL}{ {K_w^\bullet(\calL)} }
\newcommand{\PwL}{ {\Psi_w^\bullet(\calL)} }
\newcommand{\IvL}{ {I(v, \calL)} }
\newcommand{\Epsilon}{ {\calE} }
\newcommand{\Xv}{ {\calX_{(v)}} }
\newcommand{\Xvp}{ {\calX_{(v)}'} }
\newcommand{\cbphi}{\cok(\bar\varphi)}
\newcommand{\trans}{\tilde{\ }}
\newcommand{\tL}{\tilde{L}}
\newcommand{\Qc}{\Qu^c}
\newcommand{\Qpc}{\Qu_p^c}
\newcommand{\Ehat}{\hat{E}}
\newcommand{\bnh}{\hat{\beta}_n}
\newcommand{\bih}{\hat{\beta}_i}
\newcommand{\bjh}{\hat{\beta}_j}
\newcommand{\kinh}{\hat{\kappa}_{i,n}}
\newcommand{\kjnh}{\hat{\kappa}_{j,n}}
\newcommand{\kimh}{\hat{\kappa}_{i,m}}
\newcommand{\kxi}{k_{\xi,i}}
\newcommand{\kxn}{k_{\xi,n}}
\newcommand{\kxnull}{k_{\xi,0}}
\newcommand{\ahme}{\fra^{-1}}
\newcommand{\name}{(\N(\fra) - 1)}
\newcommand{\Lval}{L_S^*(\chi^{-1})}
\newcommand{\chimn}{\chi(\mu\nu)}
\newcommand{\qb}{\frq^{y+1}}
\newcommand{\pa}{\frp^{x+1}}
\gdef \projlim#1{\lim\limits_{\stackrel\la #1} }
\gdef\F#1{ {\hat{F^\prime_{#1}}} }


\newcommand{\etatwo}{\eta^{(2)}}
\newcommand{\omtl}{\widetilde{\omega}}
\newcommand{\Lul}{\underline{L}}
\newcommand{\Lp}{L^{\prime}}
\newcommand{\Lulp}{\underline{L}^{\prime}}
\newcommand{\omul}{\underline{\omega}}
\newcommand{\ul}[1]{\underline{#1}}
\newcommand{\Z}{\mathbb{Z}}

\renewcommand{\Cl}{Cl}
\renewcommand{\ll}{\left \langle}
\newcommand{\rl}{\right \rangle}

\newcommand{\Rec}{\mathrm{Rec}}
\newcommand{\rec}{\mathrm{rec}}
\newcommand{\Ord}{\mathrm{Ord}}
\newcommand{\nuL}{\tilde{\nu}}
\newcommand{\val}{val}
\newcommand{\Col}{Col}

\newcommand{\oll}[1]{\overline{ \overline{#1}}}
\newcommand{\ol}[1]{\overline{#1}}

\newcommand{\betapp}{\beta^{\prime \prime}}
\newcommand{\sigmapp}{\sigma^{\prime \prime}}
\newcommand{\betap}{\beta^{\prime }}
\newcommand{\sigmap}{\sigma^{\prime}}

\newcommand{\prebeta}{\tilde{\beta}}

\newcommand{\w}{w}
\newcommand{\vp}{w^{\prime}}
\newcommand{\vpp}{w^{\prime \prime}}
\newcommand{\prevpp}{\widetilde{w^{\prime \prime}}}

\newcommand{\Kinf}{K_{\infty}}
\newcommand{\Kn}{K_{n}}
\newcommand{\Km}{K_{m}}
\newcommand{\Konen}{K_{1,n}}
\newcommand{\Kponen}{K^{\prime}_{1,n}}
\newcommand{\Lonen}{L_{1,n}}
\newcommand{\Kzero}{K_{0}}
\newcommand{\Kp}{K^{\prime}}
\renewcommand{\L}{L}

\newcommand{\locKinf}{H_{\infty}}

\newcommand{\locKoneinf}{H_{1,\infty}}
\newcommand{\locKponeinf}{H^{\prime}_{1,\infty}}
\newcommand{\locLoneinf}{M_{1,\infty}}

\newcommand{\locKtwoinf}{H_{2,\infty}}
\newcommand{\locKptwoinf}{H^{\prime}_{2,\infty}}
\newcommand{\locLtwoinf}{M_{2,\infty}}

\newcommand{\NC}{\mathcal{N}}
\newcommand{\locKonen}{H_{1,n}}
\newcommand{\locKponen}{H^{\prime}_{1,n}}
\newcommand{\locLonen}{M_{1,n}}
\newcommand{\locLin}{M_{i,n}}

\newcommand{\locKpin}{H^{\prime}_{i,n}}
\newcommand{\locKpjn}{H^{\prime}_{j,n}}
\newcommand{\locKin}{H_{i,n}}
\newcommand{\locKjn}{H_{j,n}}

\newcommand{\locKtwon}{H_{2,n}}
\newcommand{\locKptwon}{H^{\prime}_{2,n}}
\newcommand{\locLtwon}{M_{2,n}}

\newcommand{\locKn}{H_{n}}
\newcommand{\locKm}{H_{m}}
\newcommand{\locFn}{F_{n}}
\newcommand{\locFi}{F_{i}}
\newcommand{\locKi}{H_{i}}
\newcommand{\locKzero}{H_{0}}
\newcommand{\locKp}{H^{\prime}}
\newcommand{\locL}{H}

\newcommand{\locF}{F}
\newcommand{\locFp}{F^{\prime}}
\newcommand{\locFpp}{F^{\prime \prime}}
\newcommand{\locFponen}{F^{\prime}_{1,n}}
\newcommand{\locFpin}{F^{\prime}_{i,n}}
\newcommand{\locFpponen}{F^{\prime \prime}_{1,n}}
\newcommand{\locFppin}{F^{\prime \prime}_{i,n}}
\newcommand{\locFpn}{F^{\prime}_{n}}
\newcommand{\locFppn}{F^{\prime \prime}_{n}}

\newcommand{\locLchin}{H^{\prime}_{\chi_{i,n}}}
\newcommand{\locLchi}{H_{\chi}}

\newcommand{\localbeta}{\tilde{\beta}_{i,n}}
\newcommand{\locbeta}{\tilde{\beta}_{i,n}}
\newcommand{\locbetap}{\tilde{\beta^{\prime}}_{i,n}}
\newcommand{\lockappa}{\tilde{\kappa}_{i,n}}
\newcommand{\betad}{(\beta^{\circ})}

\newcommand{\Knr}{{K_{\mathrm{nr}}}}
\newcommand{\Knrtimes}{{K_{\mathrm{nr}}^\times}}
\newcommand{\Lnr}{{L_{\mathrm{nr}}}}
\newcommand{\Lnrtimes}{{L_{\mathrm{nr}}^\times}}

\newcommand{\blue}{\textcolor{blue}}
\newcommand{\red}{\textcolor{red}}

\maketitle
\begin{abstract}
Let $L/k$ be a finite abelian extension of an imaginary quadratic number field
$k$. Let $\frp$ denote a prime ideal of $\Ok$ lying over the rational prime $p$.
We assume that $\frp$ splits completely in $L/k$ and that $p$ does not divide the class number of $k$.
If $p$ is split in $k/\Qu$ the first named author has adapted a construction of Solomon to obtain elliptic $\frp$-units in $L$.
In this paper we generalize this construction to the non-split case and obtain in this way a pair of elliptic $\frp$-units
depending on a choice of generators of a certain Iwasawa algebra (which here is of rank 2).
In our main result we express the $\frp$-adic valuations of these $\frp$-units in terms of the $p$-adic logarithm of an explicit elliptic unit.
The crucial  input for the proof of our main result is the computation of the constant term of a suitable Coleman power series, where we rely on recent work of T. Seiriki.
\end{abstract}

\noindent
{\bf MS Classification 2010:} 11R27, 11G16, 11G15
\section{Introduction}
\label{sec_intro}

The principal motivation for this paper is the study of the equivariant Tamagawa Number Conjecture
(short eTNC) for abelian extensions of an imaginary quadratic field $k$. We recall that the eTNC is proved for
abelian extensions of $\Qu$ by important work of Burns and Greither \cite{BuGr03} complemented by work of 
Flach \cite{Fl04, Fl11} who dealt with the $2$-part. We also refer the reader to \cite{Fl04} for a 
survey of the strategy as well as for details of the proof. 

For abelian extensions of an imaginary quadratic field $k$ the first named author in \cite{Bl06} adapted the proof of Burns and Greither
as presented in Flach's survey \cite{Fl04} to prove the $p$-part of the eTNC for primes $p \ge 3$
\footnote{The proof in \cite{Bl06} uses a result of Gillard 
which was at that time only proved for $p \ge 5$. However, Oukhaba and Vigui\'{e} {\cite{OuVi16}} in the meantime proved Gillard's result also in the cases $p=2,3$.}
which are split in $k/\Qu$ and coprime to the class number $h_k$ of $k$. 

The strategy of proving these results is as follows: one first proves an equivariant Iwasawa main conjecture and then 
descents to finite levels by taking coinvariants. This technique is often very complicated because of a 'trivial zeros phenomenon'. 
In this context we also refer
the reader to recent work of Burns, Kurihara and Sano \cite{BKS2}, where they describe a concrete strategy for proving special cases of the eTNC.
In particular, they formulate an Iwasawa theoretic version of a conjecture of Mazur and Rubin in \cite{MaRu16} and independently of Sano in \cite{Sa14_Comp}.
This conjecture \cite[Conj.~4.2]{BKS2} holds in the 'case of no trivial zeros' (see \cite[Cor.~4.6]{BKS2}) and can be shown for 
abelian extensions of $\Qu$ as a consequence of the main result of Solomon in \cite[Thm.~2.1]{So92}. 
Burns, Kurihara and Sano use the conjectured validity of \cite[Conj.~4.2]{BKS2} as a main tool in their descent computations.

Much earlier, Burns and Greither applied Solomon's result directly for the descent in the 'trivial zeros case' for abelian extensions of $\Qu$. 
A result similar to Solomon's has been proved by the first named author in \cite{Bl04} for abelian extensions of an imaginary quadratic field $k$
and split primes $p$ not dividing the class number $h_k$. As in the case of abelian extensions of $\Qu$, this result  was then used to perform the descent
in the {'trivial zeros case'} and is therefore a crucial ingredient in the proof of the $p$-part of the eTNC in \cite{Bl06}.

In this manuscript we consider rational primes $p$ which do not split in the imaginary quadratic field $k$. 
Let $\frp$ denote the prime ideal of $k$ lying over $p$. 
In the rest of the introduction we briefly describe our main result for primes $p \ge 5$. For the general result see 
Section \ref{sec_formulationmainthm}, in particular, Theorem \ref{thm_main}.

For any integral ideal $\frm$ of $k$ we write $k(\frm)$ for the ray class field modulo $\frm$.
Then, for any  non-trivial integral ideal $\frf$ of $k$, the extension  $k(\frf \frp^\infty)/k(\frf\frp)$ is a rank two $\Zp$-extension and it is not at all
clear what a natural generalization of Solomon's construction could be.

Let $L/k$ be a finite abelian extension and assume that $\frp$ splits completely in $L/k$. We let $\frf$ denote a multiple of the conductor
$\frf_L$ of $L$ which is prime to $\frp$ and also satisfies the assumption that the natural map $\O_k^\times \lra (\O_k/\frf)^\times$ is injective. 

We write $k_{\frp}$ for the completion of $k$ at $\frp$. If $\Gamma$ denotes the Galois group of $k(\frf\frp^\infty)/k(\frf\frp)$, the 
global Artin map induces an isomorphism $1 + \frp\O_{k_{\frp}} \stackrel{\simeq}\lra \Gamma $ and we write 
\hbox{$\chiell \colon \Gamma  \stackrel{\simeq}\lra  1 + \frp\O_{k_{\frp}}$} for
its inverse. The $p$-adic logarithm defines an isomorphism \hbox{$1 + \frp\O_{k_{\frp}}  \simeq \frp\O_{k_{\frp}}$}, so that $\Gamma$ is
free of rank $2$ over $\Zp$ since, by assumption,
$p$ is non-split in $k/\Qu$. We fix topological generators $\gamma_1, \gamma_2$ of $\Gamma$ 
and set $\omega_i = \omega_i(\gamma_i) := \log_p(\chiell(\gamma_i))$  for $i \in \{1,2 \}$.  Then the set $\{ \omega_1, \omega_2 \}$ is a $\Qp$-basis of $k_\frp$ and we write
$\pi_{\omega_i} \colon k_{\frp} \lra \Qp$ for the projection to the $\omega_i$-component.

For an integral ideal $\frm$ we will use special values of an elliptic function defined by Robert in \cite{Ro90} 
to define elliptic units $\psi_\frm$ in the ray class field $k(\frm)$ (for a proper definition see Section~\ref{subsec_ellipticunits}).
These elliptic units satisfy the usual norm relations, so that a construction motivated by Solomon's work gives rise to a pair of  
elliptic $\frp$-units  $\kappa_1(\gamma_1), \kappa_2(\gamma_2)$ which depend
upon our fixed choice of generators. The precise definition of $\kappa_i(\gamma_i)$ will be given in Section \ref{sec p-units}.

Let $S_\frp = S_\frp(L)$ denote the set of primes of $L$ above $\frp$  and  write $\O_{L, S_{\frp}}^\times$ for the group of $S_\frp$-units of $L$.
Our elements  $\kappa_1(\gamma_1), \kappa_2(\gamma_2)$ are actually not elements in $\O_{L, S_{\frp}}^\times$, but are contained in $\O_{L, S_{\frp}}^\times \tensor_\Ze \Zp$.
Let $v_\frP$ denote the normalized valuation of $L$ at $\frP$, but we also write  $v_\frP \colon \O_{L, S_{\frp}}^\times \tensor_\Ze \Zp \lra \Zp$ for the $p$-completion of $v_\frP$.

Since $\frp$ splits completely in $L/k$, each $\frP \in S_\frp$
corresponds to a unique embedding $j_\frP$ of $L$ into $k_\frp$. 

We are finally in a position to formulate our main result. 
The theorem is restated in Theorem \ref{thm_main} where our partial results for  $p=2,3$ are included.

\begin{theorem_wn}\label{thm_intro}
Let $p \ge 5$ be a prime that does not split in $k / \Qu$ and let $\frp$ be the prime ideal of $k$ above $p$. 
Let $L$ be a finite abelian extension of $k$ and assume that $\frp$ splits completely in $L/k$. In addition, we assume
that $p$ does not divide the class number $h_k$. 
Then we have the following equality in $\Zp$ for each $\frP \in S_{\frp}$:
\begin{eqnarray*}
\pi_{\omega_{i}}\left( \log_p(j_{\frP}(\N_{k(\frf)/L} (\psi_{\frf}))) \right) = v_\frP (\kappa_{j}),
\end{eqnarray*}
where  $i,j \in \{1,2\}$, $i \ne j$.
\end{theorem_wn}

\begin{remark}
  The hypothesis that $\frp$ splits completely in $L/k$ is no obstruction for the intended applications to the eTNC
since in the 'trivial zeros case' this condition is automatically fulfilled.
\end{remark}

We briefly describe the structure of the proof of Theorem \ref{thm_main}. We set \hbox{$F:=k(\frf)$}. 
By the theory of complex multiplication there exists an elliptic curve $E$ over $F$ such that $E$ has complex multiplication 
by $\O_{k}$ and $F(E_{tor})/k$ is abelian. The associated formal group $\hat{E}$ is a relative Lubin-Tate  group and the elliptic units
$u_n := \psi_{\frf\frp^{n+1}}$ define a norm-coherent sequence $u = (u_n)_{n \ge 0}$ of units in the Iwasawa tower of Lubin-Tate extensions. 
The essential input for the proof of Theorem \ref{thm_main} is the determination of the constant term $\Col_u(0)$ of the 
associated Coleman power series $\Col_u$.

In the situation of \cite{Bl04},  where $p$ is split, the formal group $\hat{E}$ is of height $1$ (defined over $\Qp$) and one could
to a large extend follow Solomon's approach replacing the multiplicative group $\mathbb{G}_m$ by $\hat{E}$. 
If $p$ does not split, $\hat{E}$ is of height $2$ defined over the quadratic extension $k_\frp$ of $\Qp$ and the situation is more complicated.

For the computation of the constant term $\Col_{u}(0)$ we now apply a recent result of T. Seiriki in \cite{Sei15}, who describes  $\Col_u(0)$ in terms of a pairing 
whose definition is closely related to Solomon's construction. For the convenience of the reader we give a
self-contained proof of Seiriki's result in Section \ref{sec_constantterm}. 
We follow entirely his strategy but adapt some of his arguments and fill in details
wherever we felt that this is necessary.

In future work the second named author will study the relation of our main result to the Iwasawa theoretic Mazur-Rubin-Sano conjecture \cite[Conj.~4.2]{BKS2} and the relation to the eTNC.

The authors want to take the opportunity to express their gratitude for the wonderful organization of Iwasawa 2017.
The first named author also wants to thank the organizers of Iwasawa 2017 for the invitation to give a talk at the conference. 
He also wants to thank Cristian Popescu for many fruitful discussions on the subject of this paper.

\paragraph{Notation}

For any field $F$ we write $F^{c}$ for a fixed algebraic closure of $F$. For a { finite extension $E/F$ of fields we write $\N_{E/F}$ for the field theoretic norm.}
For a number field (resp. $p$-adic field) $F$ we write $\O_F$ for the ring of integers (resp. the valuation ring). 
If $F$ is a $p$-adic field, we let $\frp_F$ denote the maximal ideal of $\O_F$.

Let $k$ {always} denote an imaginary quadratic field.  If $\frf$ is an ideal of $\O_{k}$, we write $k(\frf)$ for the ray class field of conductor $\frf$. 
In particular, $k(1)$ is the Hilbert class field. We let $w(\frf)$ denote the number of roots of unity congruent to $1$ modulo $\frf$. 
Hence $w(1)$ equals the number $w_{k}$ of roots of unity in $k$. Moreover, we use the notation $\N(\frf)$ for the ideal norm. 
  If $F/k$ is an abelian extension we write $\frf_{F}$ for the conductor of $F$. For an integral ideal $\frc$ 
which is relatively prime to $\frf_{F}$, 
we write  $\sigma(\frc)$ or $\sigma_{\frc}$ or $(\frc, F/k)$ for the associated Artin automorphism.

Let $p$ be a prime. If $H$ is a finite extension of $\Qp$, we write $\rec_H \colon H^\times \lra \Gal(H^{ab}/H)$ for the reciprocity map of local class field theory, { where }$H^{ab}$ denotes the maximal abelian extension of $H$ in $H^c$.
Moreover, we denote by $\log_{p}$ the Iwasawa logarithm normalized such that $\log_{p}(p)=0$.

\section{Formulation of the Main Theorem}
\label{sec_formulationmainthm}

\subsection{Elliptic units}
\label{subsec_ellipticunits}

In this subsection we introduce the notion of elliptic units that we will use in this paper. Our elliptic units are special values of
a meromorphic function $\psi$ which was introduced and studied by Robert in \cite{Ro90, Ro92}. 

 We let  $L \subseteq \Ce$ denote a $\Ze$-lattice of rank $2$ with complex multiplication by $\Ok$. 
For any integral  $\Ok$-ideal  $\fra$ satisfying $(\N(\fra),6)=1$ we define a meromorphic function
 \[
  \psi(z; L , \fra):= \tilde{F}(z; L, \fra^{-1} L), \quad z \in \Ce,
 \]
where $\tilde{F}$ is defined in \cite[Th\'eor\`eme principal, (15)]{Ro90}.
This function $\psi$ is closely related to the function $\Theta_{0}(z;\fra)$ used by Rubin in \cite[Appendix]{Ru87} 
and it is the $12$-th root of the function $\Theta(z;L,\fra)$ defined in
\cite[Ch.~II.2]{dS87}. 

Let $\frm$ denote an integral ideal  such that  $(\frm , \fra)=1$.
{Then the element $\psi_\frm$ used in the introduction is defined by $ \psi_\frm := \psi(1; \frm , \fra)$.}

We refer the reader to \cite[Sec.~2]{Bl04}, where the basic properties of these elliptic units are stated and proved. In particular,
we recall that special values of $\psi$ satisfy norm relations analogous to those satisfied by cyclotomic units (see \cite[Thm.~2.3]{Bl04}).

\subsection{Construction of elliptic $\frp$-units}\label{sec p-units}

Let $L$ denote a finite abelian extension of the imaginary quadratic field $k$. We fix a prime ideal $\frp$ of $\Ok$ above a rational prime $p$. 
We write $H$ for the completion $k_\frp$ of $k$ at $\frp$. For primes $p \ge 5$ we will have no further assumptions on
$\frp$ and $k$ besides hypotheses (H1) and (H2) below, however, for $p=2$ or $p=3$ we need to impose the following conditions.

\begin{itemize}
\item If $p=2$ we assume that either a2) or b2) holds:
  \begin{itemize}
  \item [a2)] $p$ is split in $k$.
  \item [b2)] $p$ is ramified in $k$ and $H = \Qu_2(\zeta_4)$, where $\zeta_4$ is a primitive fourth root of unity in $\Qu_2^c$.
  \end{itemize}
\item If $p=3$ we assume that either a3), b3) or c3) holds:
  \begin{itemize}
  \item [a3)] $p$ is split in $k$.
  \item [b3)] $p$ is inert in $k$. 
  \item [c3)] $p$ is ramified in $k$ and $H = \Qu_3(\zeta_3)$, where $\zeta_3$ is a primitive third root of unity in $\Qu_3^c$.
  \end{itemize}
\end{itemize}

We summarize these conditions in  the following table. Here the integer $s$ is defined to be the smallest integer such that $s > e/(p-1)$ with $e$ denoting
the ramification index of $p$ in $k/\Qu$. 
\[
\centering
  \begin{tabular}{c|c|c|c}
    & split & inert & ramified \\
\hline
2 & $s=2$ & --- & $H = \Qu_2(\zeta_4)$, $s=3$ \\
\hline
3 & $s=1$ & $s=1$ & $H = \Qu_3(\zeta_3)$, $s=2$ 
  \end{tabular}
\]
\pagebreak

Our main hypotheses are as follows.

\begin{hypotheses}
\label{assumption_inert_ray}
\leavevmode
\begin{enumerate}
\item[(H1)] $\frp$ splits completely in $L$.
\item[(H2)] $p$ does not divide the class number of $k$.
\end{enumerate}
\end{hypotheses}

Let $\frf_L$ be the conductor of $L$ and fix an integral ideal $\frf$ of $\Ok$ such that $\frf_{L} \mid \frf$, $\frp \nmid \frf$ and $w(\frf)=1$.
Note that for $p \ge 5$ we have $s = 1$. 
We set  
\begin{align*}
F:= k(\frf),   \quad k(\frp^\infty) :=  \bigcup_{n\ge 0} k(\frp^{n}) \quad\text{ and }\quad K_\infty :=  \bigcup_{n\ge 0} k(\frf\frp^{s+n}).
\end{align*}
We write $T$ for the torsion subgroup of $\Gal(k(\frp^\infty)/k)$ and let $k_\infty := k(\frp^\infty)^T$ be the
fixed field of $T$. Then $k_\infty/k$ is a $\Ze_p^d$-extension with $d=1$ if $p$ splits  in $k$ and $d = 2$ otherwise. By (H1)  $\frp$ is unramified in
$F$, and thus  (H2) implies $F \cap k_\infty = k$. 

We now investigate the extension $K_\infty/F$. We set $F_\infty := Fk_\infty$ and $L_{\infty}:=L k_{\infty}$. Since $\Gal(K_\infty/F_\infty)$ is finite and $\Gal(F_\infty/F)$ is torsion-free
we see that $\Gal(K_\infty/F_\infty)$ is the torsion subgroup of $\Gal(K_\infty/F)$. By class field theory we obtain
\[
\Gal(K_\infty/F) \simeq \varprojlim_{n} \left( \O_k / \frp^{s+n} \right)^\times \simeq \O_H^\times = \mu_H' \times (1+\frp_H)
\]
with $\mu_H' $ denoting the roots of unity of order $\N(\frp) - 1$ in $H$. With our definition of $s$ the $p$-adic exponential $\exp_p$ converges
on $\frp_{H}^s$, so that $1+\frp_H^s$ is torsion-free. 

\begin{lemma}
With $s$ as above we set $K_0 := k(\frf\frp^s)$. Then
  $K_0 \cap F_\infty = F$ and $K_0F_\infty = K_\infty$.
\end{lemma}

\begin{proof}
  In each case one can show that 
\[
\left| \left( 1 + \frp_H \right)_{tors} \right| = \left| \frac{1 + \frp_H}{1 + \frp_H^s}  \right|,
\]
where we write $\left( 1 + \frp_H \right)_{tors}$ for the subgroup of torsion elements of $1 + \frp_H$.
For $s=1$ this is trivial and, for example, if $p = 2 $ is ramified and $H = \Qu_2(\zeta_4)$, then $\left| {(1 + \frp_H)}/{(1 + \frp_H^s)}  \right| = 4$ and
$\left( 1 + \frp_H \right)_{tors} = \langle \zeta_4 \rangle$.

It is then easily shown that $(1 + \frp_H^s) \times \left( 1 + \frp_H \right)_{tors} = 1 + \frp_H$. In summary, we have a direct product
decomposition
\[
\Gal(K_\infty/F) \simeq \mu_H' \times  \left( 1 + \frp_H \right)_{tors} \times (1 + \frp_H^s)
\]
and the lemma follows because the fixed field of $1 + \frp_H^s$ (resp. $ \mu_H' \times  \left( 1 + \frp_H \right)_{tors}$) is $k(\frf\frp^s) = K_0$
(resp. $F_\infty$).
\end{proof}

In each case we thus obtain the following diagram of fields (with $K_n, F_n, L_n$ and $k_n$ defined below)
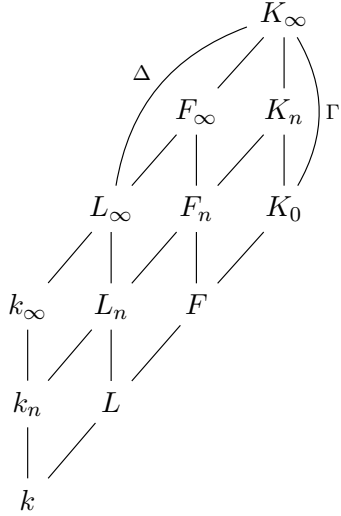
\[
\begin{tikzcd}[row sep=normal, column sep=tiny]
 {} & &  & K_{\infty} \arrow[-, bend left]{dd}{\Gamma} \arrow[-]{d} \arrow[-]{ld}  \arrow[-, bend right, swap]{ddll}{\Delta}\\
 &  & F_{\infty} \arrow[-]{ld} \arrow[-]{d} & K_{n}  \arrow[-]{d} \arrow[-]{ld} \\
 & L_{\infty} \arrow[-]{ld} \arrow[-]{d} & F_{n} \arrow[-]{dl} \arrow[-]{d} & K_{0} \arrow[-]{ld}  \\
k_{\infty} \arrow[-]{d} &  L_{n} \arrow[-]{ld} \arrow[-]{d} & F \arrow[-]{dl} & \\
 k_{n} \arrow[-]{d} & L \arrow[-]{dl} &  &\\
 k & & &
\end{tikzcd}
\]
We set $\Gamma:=\Gal(K_\infty/K_0)$ and identify $\Gamma$ via restriction with each of the Galois groups
$\Gal(F_\infty /F_{0}), \Gal(L_{\infty}/L_{0}) \text{ and } \Gal(k_{\infty}/k_{0})$. We let $K_n, F_n, L_n$ and $k_n$ denote the fixed field of $\Gamma^{p^n}$
of $K_\infty, F_\infty, L_\infty$ and $k_\infty$, respectively.
The case where $p$ is split in $k/\Qu$ is already treated in \cite{Bl04}. We therefore assume from now on that $p$ is non-split and write $e$ for the ramification degree of $p$ in $k/\Qu$.
It is not difficult to see that $K_n = k(\frf\frp^{s+en})$ for all $n \ge 0$.
We  have an isomorphism $\Gamma \simeq \Ze_p^2$ given by the following composition 
\begin{equation}\label{chiell def}
\Gamma \stackrel{\chiell}{\lra} 1 + \frp_H^s \stackrel{\log_p}\lra \frp_H^s \simeq \Ze_p^2.
\end{equation}
where $\chiell$ is induced by the inverse of the global Artin map.
Finally, we fix topological generators $\gamma_1$ and $\gamma_2$ of $\Gamma$ and by abuse of notation 
also write $\gamma_i$, $i=1,2$, for each of the restrictions of $\gamma_i$ to $K_n, F_n, L_n$ or $k_n$.

We fix an auxiliary integral ideal $\fra$ of $\Ok$ such that $(\fra , 6 \frf \frp)=1$.

\begin{definition}\label{psi_n def}
Let $i \in \{1,2\}$ and set $\Delta := \Gal(K_\infty / L_\infty) \simeq \Gal(K_0/L_0)$.  We define 
\[
K_{i,n} := K_{n}^{\langle \gamma_i \rangle} \quad L_{i,n}:= K_{i,n}^{\Delta}
\]
and set
\[
 \epsilon_{i,n} := \begin{cases} \N_{K_{n}/L_{i,n}}(\psi(1; \frf \frp^{s+n}, \fra)), & \text{ if }  \frp { \text{ is inert in } k/\Qu,}\\
\N_{K_{n}/L_{i,n}}(\psi(1; \frf \frp^{s+2n}, \fra)), & \text{ if } \frp {\text{ is ramified in } k/\Qu.}
                           \end{cases}
\]
\end{definition}

The groups $\Gal(K_{i,n}/K_0) \simeq \Gal(L_{i,n}/L)$ are cyclic groups of order $p^n$ 
generated by the image of $\gamma_j$ where $j \in \{1,2\}, j \ne i$.

So we obtain the following diagram of fields.
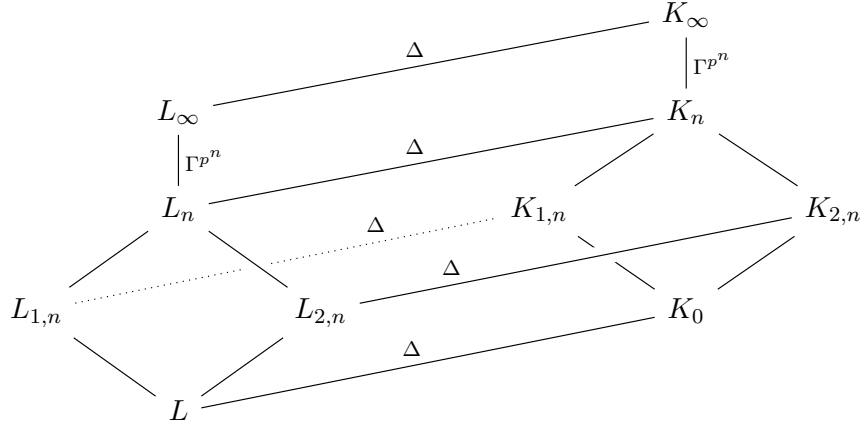
\[
\begin{tikzcd}
{}&  &  &  &  & K_{\infty} \arrow[d, "\Gamma^{p^{n}}", no head] &  \\
 & L_{\infty} \arrow[rrrru, "\Delta", no head] \arrow[d, "\Gamma^{p^{n}}", no head] &  &  &  & K_{n} \arrow[rd, no head] &  \\
 & L_{n} \arrow[rrrru, "\Delta", no head] \arrow[ld, no head]  &  &  & K_{1,n} \arrow[ru, no head] \arrow[rd, no head] &  & K_{2,n} \arrow[ld, no head] \\
L_{1,n} \arrow[rd, no head] \arrow[rrrru, "\Delta" near end, no head, dotted] &  & L_{2,n} \arrow[rrrru, "\Delta" near start, crossing over, no head]  \arrow[lu, no head, crossing over]&  &  & K_{0} &  \\
 & L \arrow[ru, no head] \arrow[rrrru, "\Delta", no head] &  &  &  &  & 
\end{tikzcd}
\]


From the norm relations of \cite[Thm.~2.3]{Bl04} we obtain for $i \in \{1,2\}$ and  $n \ge m \ge 0$
\begin{align*}
\N_{L_{i,n}/L_{i,m}}(\epsilon_{i,n}) =\epsilon_{i,m} \quad \text{ and } \quad
\N_{L_{i,n}/L}(\epsilon_{i,n}) =1,
\end{align*}
where the second equality holds since $\sigma(\frp)|_L = 1$ by hypothesis (H1).

By Hilbert's Theorem 90 there exist for $i,j \in \{1,2\}$  with $i \ne j$  unique elements 
$\beta_{i,n}\in~\left( L_{i,n}\right)^\times / L^\times$  such that 
$$\beta_{i,n}^{\gamma_j - 1} = \epsilon_{i,n}.$$
We have to keep in mind, that the $\beta_{i,n}$ depend on $\frf$, the choice of $\gamma_1$ and $\gamma_2$, and the choice of 
the auxiliary ideal $\fra$.

\begin{definition}
\label{def_general_kappa_n}
We define for $i \in \{1,2\}$ 
\[
\kappa_{i,n} =
 \kappa_{i,n}(L,\gamma_1, \gamma_2, \frf, \fra)   := \N_{L_{i,n}/L} (\beta_{i,n}) \in L^\times / {\left( L^\times \right)}^{p^n}.
\]
\end{definition}

For a prime $\frq$ of a number field $N$ we write $v_\frq\colon N^\times \lra \Ze$ for the normalized valuation map.

\begin{lemma}\cite[cf. Prop.~2.2]{So92}\label{lemma_valzerocoprimeidealmodpn}
Let $\frq$ be a prime ideal of $L$ relatively prime to $\frp$. Then
\[
v_\frq(\kappa_{i,n}) \equiv 0 (\mod p^n\Ze).
\]
\end{lemma}

As in \cite[Lemma~2.3]{So92} it is not difficult to see that
if $m \geq n \geq 0$ then the natural quotient map 
\begin{equation}\label{nat quot map}
L^{\times}/\left(L^{\times} \right)^{p^{m}} \rightarrow L^{\times}/ \left( L^{\times} \right)^{p^{n}}
\end{equation}
takes $\kappa_{i,m}$ to $\kappa_{i,n}$.

\begin{definition}
\label{def_general_kappa}
For $i \in \{1,2\}$ we define
\[
\kappa_{i} :=
 \kappa_{i}(L,\gamma_1, \gamma_2, \frf, \fra) := 
(\kappa_{i,n})_{n=0}^{\infty} \in \varprojlim_{n} L^{\times} / \left( L^{\times}\right)^{p^{n}}.
\]
\end{definition}

Recall that $\calO_{L, S_\frp}^\times$ denotes the group of $S_\frp$-units of $L$ with $S_{\frp}= S_{\frp}(L)$
denoting the set of prime ideals of $L$ above $\frp$. We then
have a natural injection
\[
\calO_{L, S_\frp}^\times \tensor_{\Ze} \Zp\simeq 
\varprojlim_{n} \left.\O_{L, S_{\frp}}^{\times} \middle/ \left( \O_{L, S_{\frp}}^{\times} \right)^{p^{n}} \right. 
\longrightarrow \varprojlim_{n} L^{\times} / \left( L^{\times}\right)^{p^{n}}.
\]

\begin{proposition}{\cite[cf. Prop.~2.3]{So92}}
For $i=1,2$ we have $\kappa_{i} \in \O_{L, S_{\frp}}^{\times} \otimes_{\Ze} \Ze_{p}$. 
\end{proposition}
\begin{proof}
We fix $i$ and set $\kappa := \kappa_i$ and $\kappa_n := \kappa_{i,n}$. We recall that $\kappa_n \in L^\times / (L^\times)^{p^n}$ and note
that it suffices to show  $\kappa_n \cap \calO_{L, S_\frp}^\times \ne \emptyset$ for all $n \ge 0$. 

Let $m$ be the order of the $S_\frp$-class group $cl_{L, S_\frp}$ and recall that
\[
cl_{L, S_\frp} \simeq I_L / \langle \frP_1, \ldots, \frP_r, P_L \rangle,
\]
where $I_L$ denotes the group of fractional ideals of $L$, $P_L$ the subgroup of principal ideals and $\frP_1, \ldots, \frP_r$ the prime ideals
of $L$ lying over $\frp$. Write $m = p^tm'$ with a natural number $t$ and $p \nmid m'$ and choose $c \in \kappa_{n+t}$. 
By Lemma \ref{lemma_valzerocoprimeidealmodpn} we obtain
\[
  (c) = \frP_1^{a_1} \cdots \frP_r^{a_r} \fra^{p^{n+t}}
\]
with integers $a_1, \ldots, a_r$ and a fractional ideal $\fra$ which is coprime to $\frp$. Hence 
\hbox{$\fra^{p^{n+t}} \in \langle \frP_1, \ldots, \frP_r, P_L \rangle$}. Clearly $\fra^{m} \in \langle \frP_1, \ldots, \frP_r, P_L \rangle$ and hence
\hbox{$\fra^{p^t} \in \langle \frP_1, \ldots, \frP_r, P_L \rangle$}. Therefore, there exists an element $x \in L^\times$ and integers  $b_1, \ldots, b_r$ 
such that
\[
\fra^{p^t} = \frP_1^{b_1} \cdots \frP_r^{b_r}x.
\]
We conclude that
\[
(c) = \frP_1^{a_1 + p^nb_1} \cdots \frP_r^{a_r + p^nb_r} x^{p^n}.
\]
It follows that $cx^{-p^n} \in \calO_{L, S_\frp}^\times$ and since the natural quotient map in (\ref{nat quot map}) takes
$\kappa_{n+t}$ to $\kappa_n$  we obtain $cx^{-p^n} \in \kappa_n$.
\end{proof}

\subsection{Statement of the Main Theorem}
For each prime $\frP$ of $L$ above $\frp$ the valuation map $v_\frP \colon L^\times \lra \Ze$ induces a natural homomorphism, also denoted by $v_\frP$,
\[
v_{\frP} \colon  \varprojlim_{n} L^\times / \left( L^\times \right)^{p^{n}} \rightarrow \Ze_{p}.
\]
The restriction of $v_\frP$ to $\calO_{L, S_\frp}^\times \tensorZ \Zp$ obviously coincides with the $\Zp$-linear extension of 
$v_\frP \colon \calO_{L, S_\frp}^\times \lra \Ze$.

By our assumption (H1) each prime $\frP \in S_\frp$ corresponds to a unique embedding
\[
j_\frP \colon L \hookrightarrow k_\frp.
\]
Recalling the definitions of the isomorphisms in (\ref{chiell def}) we set
\[
\omega_i := \log_p(\chiell(\gamma_i))
\]
for $i \in \{1,2\}$
 and obtain a $\Zp$-basis $\omega_1, \omega_2$ of $\frp_{H}^s$. The set $\{ \omega_1, \omega_2 \}$ is also a $\Qp$-basis of $H$ and we
write $\pi_{\omega_i} \colon H \lra \Qp$ for the projection maps. Explicitly, we have for each $\alpha \in H$ the equality 
$\alpha = \pi_{\omega_1}(\alpha)\omega_1 + \pi_{\omega_2}(\alpha)\omega_2$.

Now we are ready to formulate our main theorem:
\begin{theorem}
\label{thm_main}
Let $p$ be a prime which does  not split in $k / \Qu$ and assume $\mathrm{(H1)}$ and $\mathrm{(H2)}$. 
If $p=2$ (resp. $p=3$) we assume in addition that b2) (resp. b3) or c3)) holds.
Then for each prime ideal $\frP$ in $L$ above $\frp$ we have
\begin{equation}\label{main equality}
\pi_{\omega_{i}}\left( \log_p(j_{\frP}(\N_{k(\frf)/L} (\psi(1;\frf,\fra)))) \right) 
= v_\frP (\kappa_{j})
\end{equation}
in $\Zp$, where $i,j \in \{1,2 \}$ with $i \neq j$.
\end{theorem}

\begin{remark}
Theorem \ref{thm_main} is the analogue of \cite[Thm.~1.1]{Bl04} which covers the case of split primes.
\end{remark}

\begin{remark}
\label{rem_equivalencemainconj 1}
The equality (\ref{main equality}) in Theorem \ref{thm_main} is equivalent to
\[
\log_p(j_{\frP}(\N_{k(\frf)/L} (\psi(1;\frf,\fra))))  = v_\frP (\kappa_{2}) \omega_1 + v_\frP (\kappa_{1}) \omega_2.
\]
\end{remark}

It is often convenient to work on finite levels. 

\begin{remark}
\label{rem_equivalencemainconj 2}
The equality (\ref{main equality}) of Theorem \ref{thm_main}  is valid, if and only if for all $n \ge 1$ and all $i,j \in \{1,2\}$ with $i \ne j$ 
the following congruence holds: 
\begin{eqnarray*}
\pi_{\omega_{i}} \left( \log_p\left( j_{\frP}(\N_{k(\frf)/L} \psi(1;\frf,\fra)) \right) \right) 
\equiv v_\frP (\kappa_{j,n}) \pmod{p^n}.
\end{eqnarray*}
\end{remark}

\section{Computing the constant term of the Coleman power series}\label{sec_constantterm}

In this section we recall a recent result of T.~Seiriki  \cite{Sei15}  who describes the constant term of 
a Coleman power series  in terms of a pairing which
is defined in the local setting of Lubin-Tate extensions in a way similar to
Solomon's construction.  We slightly have to adapt Seiriki's result for the setting of relative Lubin-Tate extensions.
For the  convenience of the reader we give a self-contained proof
of Seiriki's result in this section. We follow his strategy but adapt some
of his arguments and fill in some details whenever we feel that this is necessary.

\subsection{Definition and basic properties of Seiriki's pairing}
Let $K$ be a local field and  $L/K$ a finite abelian Galois extension with Galois group $G$.
Let $v_{K}$ be the normalized valuation of $K$ (i.e., $v_K(K^\times) = \Ze$).
We put $U_{L/K}:= \ker(\N_{L/K})$ and write $\hat{G}:= \Hom(G, \Qu/\Ze)$ for the group of characters of $G$.

\begin{definition}
\label{def_pairing}
For a character $\chi \in \hat{G}$ we set $K_{\chi}:= L^{\ker(\chi)}$. Then $G_\chi := \Gal(K_\chi/K)$ is a cyclic group whose order we denote
by $d_\chi$ (so $d_\chi = \ord(\chi)$). Let $\sigma \in G$ be such that $\chi(\sigma) = 1/d_\chi+\Ze$.
For each element  $u \in U_{L/K}$ we define $u_{\chi}:= \N_{L / K_{\chi}}(u)$
and observe that $\N_{K_{\chi}/K}(u_{\chi})=1$. Therefore, by Hilbert's Theorem 90, there exists an element
$b_{\chi} \in K^{\times}_{\chi}$ such that
\[
u_{\chi}= b_{\chi}^{\sigma -1}
\]
which is unique up to elements of $K^{\times}$. We then define 
\begin{align*}
(\cdot, \cdot)_{L/K}: U_{L/K} \times \hat{G} &\rightarrow \Qu / \Ze, \\
(u,\chi) &\mapsto v_{K}(b_{\chi}) + \Ze.
\end{align*}
\end{definition}

The pairing is obviously multiplicative in the first variable. Moreover, from the very definition we obtain
\[
(u,  \chi)_{L/K} = (N_{L/K_{\chi}}(u), \chi)_{K_{\chi}/K}.
\]

The proof of multiplicativity in the second variable is more involved. In the following we give an expanded and corrected version of Seiriki's  proof of \cite[Prop.~2.2]{Sei15}.

For a finite extension $L/K$ we write $f_{L/K}$ for the degree of the residue class field extension.

\begin{lemma}\label{exp formula}
  Let $L/K$ be a finite  abelian extension of local fields with $G = \Gal(L/K)$. Then for  $\alpha \in L^\times$, $\tau \in G$ and
$\chi \in \hat{G}$ one has
\[
\left( \frac{\tau(\alpha)}{\alpha}, \chi \right)_{L/K} = f_{L/K} v_L(\alpha) \chi(\tau)  = [L:K] v_K(\alpha)\chi(\tau) .
\]  
\end{lemma}

\begin{proof}
By our definitions $\langle \sigma|_{K_\chi} \rangle = G_\chi$, so that we may fix $s \in \Ze_{>0}$ such that
$\sigma^s \mid_{K_\chi} = \tau \mid_{K_\chi}$. 

If we set $\alpha_\chi := N_{L/K_\chi}(\alpha)$, then 
\[
N_{L/K_\chi}\left( \frac{\tau(\alpha)}{\alpha} \right) = \frac{\tau(\alpha_\chi)}{\alpha_\chi} = \frac{\sigma^s(\alpha_\chi)}{\alpha_\chi} = 
\prod_{j=0}^{s-1}  \frac{\sigma(\sigma^j(\alpha_\chi))}{\sigma^j(\alpha_\chi)}
\]
and we obtain from the definition of the pairing and multiplicativity in the first variable
\begin{eqnarray*}
\left( \frac{\tau(\alpha)}{\alpha}, \chi \right)_{L/K} 
= \sum_{j=0}^{s-1}  \left( \frac{\sigma(\sigma^j(\alpha_\chi))}{\sigma^j(\alpha_\chi)}, \chi \right)_{K_\chi/K} 
= \sum_{j=0}^{s-1}  v_K( \sigma^j(\alpha_\chi)) 
= s  v_K( \alpha_\chi).
\end{eqnarray*} 
Let $\pi_L$ be a uniformizing element in $L$ and set $\pi_\chi := N_{L/K_\chi}(\pi_L)$. 
Then $v_{K_\chi}(\pi_\chi) = f_{L/K_\chi}$.
We write $\alpha = \pi_L^a \beta$ with $a = v_L(\alpha)$ and $\beta \in \calO_L^\times$.
Then  
\begin{eqnarray*}
v_K(\alpha_\chi) =  av_K(\pi_\chi) = af_{L/K_\chi} \frac{1}{e_{K_\chi/K}} = a f_{L/K}\frac{1}{e_{K_\chi/K}f_{K_\chi/K}} = a f_{L/K}\frac{1}{d_\chi}.
\end{eqnarray*} 
Hence we obtain
\[
\left( \frac{\tau(\alpha)}{\alpha}, \chi \right)_{L/K} = v_L(\alpha) f_{L/K} \frac{s}{d_\chi}
\]
and noting $\chi(\tau) = \chi(\sigma^s) = \frac{s}{d_\chi} + \Ze$ the first equality of the lemma follows. The second equality
is then immediate from $v_L = e_{L/K}v_K$ and $e_{L/K}f_{L/K} = [L:K]$.
\end{proof}

\begin{proposition}\label{pairing mult}
  Assume that $L/K$ is a totally ramified finite abelian extension of a $p$-adic field $K$ with $G = \Gal(L/K)$. 
Then for any $u \in U_{L/K}$ and any $\chi, \chi' \in \hat{G}$ one has
\[
(u, \chi\chi')_{L/K} = (u, \chi)_{L/K} + (u, \chi')_{L/K}.
\]
\end{proposition}

\begin{proof}
  Let $\Knr$ denote the maximal unramified extension of $K$. Then $\Lnr := L\Knr$ is the maximal unramified extension of $L$ and we
may identify $\Gal(\Lnr/\Knr)$ with  $G$ via restriction. By \cite[Ch.~V.4, Prop.~7]{Ser79} the norm map 
$\N_{\Lnr/\Knr} \colon L_{\mathrm{nr}}^\times \lra K_{\mathrm{nr}}^\times$ is surjective. By  \cite[Ch.~X.7, Prop.~11]{Ser79} the $G$-module
$L_{\mathrm{nr}}^\times$ is cohomologically trivial, in particular, $\hat{H}^{-1}(G, L_{\mathrm{nr}}^\times) = 0$.

Hence there exist elements $\alpha_1, \ldots, \alpha_r \in \Lnrtimes$ and automorphisms $\sigma_1, \ldots, \sigma_r \in G$ such that
$u = \prod_{i=1}^r \frac{\sigma_i(\alpha_i)}{\alpha_i}$.
Set $L' := L(\alpha_1, \ldots, \alpha_r)$ and $K' := L' \cap \Knr$. Then
\[
(u, \chi)_{L/K} = (u, \chi)_{L'/K'} = \sum_{i=1}^r \left( \frac{\sigma_i(\alpha_i)}{\alpha_i}, \chi \right)_{L'/K'}
\]
and Lemma \ref{exp formula} implies $(u, \chi)_{L/K} = \sum_{i=1}^r v_{L'}(\alpha_i) \chi(\sigma_i)$.
Replacing $\chi$ by $\chi'$ and $\chi\chi'$, respectively,  we obtain similar expressions for $(u, \chi')_{L/K}$ and $(u, \chi\chi')_{L/K}$
and the result follows.
\end{proof}

\begin{lemma}
\label{lemma_translateunramified}
Let $L/K$ be a totally ramified finite abelian extension of degree $n$ and let $\{u_1, \ldots, u_r\} \subseteq \O_{K}^{\times}$
be a finite set of units in $K$. Let $K'/K$ be the unramified extension of degree $n$ and put $L' := K'L$. Then there exist
units $u_1^{\prime}, \ldots,  u_r^{\prime} \in \O_{L^{\prime}}^{\times}$ such that 
$u_i= \N_{L^{\prime}/K^{\prime}}(u_i^{\prime})$ for $i = 1, \ldots, r$.
\end{lemma}

\begin{proof}
It suffices to prove the lemma for $r=1$. In this case we can simply follow the proof of 
\cite[Lemma~2.4]{Sei15}. We briefly recall the argument.

We write $G_0(L/K)$ and $G_0(L'/K')$ for the inertia subgroups of $L/K$ and $L'/K'$, respectively. 
Then 
\[
G_{0}(L' / K') \stackrel{\simeq}\longleftarrow \O_{K'}^{\times}/ \N_{L'/K'}(\O^{\times}_{L'})
\stackrel{\N_{K'/K}}\longrightarrow \O_{K}^{\times}/ \N_{L/K}(\O^{\times}_{L})  \stackrel{\simeq}\longrightarrow G_{0}(L / K),
\]
where the left and the right isomorphisms are induced by the reciprocity maps of local class field theory. Furthermore,
the middle map is surjective by \cite[Ch. V.2]{Ser79}, hence it is actually an isomorphism of groups of order $n$. Since 
$\N_{K'/K}(u)=u^{n} \in \N_{L/K}(\O^{\times}_{L})$ we deduce $u\in~\N_{L' / K'}(\O_{L'}^{\times})$.
\end{proof}



\begin{proposition}\label{main prop}
Let $L/K$ be a totally ramified finite abelian  extension with $G:= \Gal(L/K)$. 
 Assume that $u \in U_{L/K}$ satisfies $(u, \chi)_{L/K}= 0$  for all characters $ \chi \in \hat{G}$.
Then there exist
\begin{enumerate}[label=\alph*)]
\item a finite unramified extension $K^{\prime}$ of $K$,
\item an integer $r$,
\item units $\betap_{1}, \ldots , \betap_{r} \in \O_{L^{\prime}}^{\times}$ with $L^{\prime}:=L K^{\prime}$ and
\item $\sigmap_{1}, \ldots , \sigmap_{r} \in \Gal(L^{\prime}/K^{\prime})$ 
\end{enumerate}
such that
\[
u = \prod_{i=1}^{r} (\betap_{i})^{\sigmap_{i}-1}.
\]
\end{proposition}

\begin{proof}
We follow the proof of \cite[Lemma~2.5]{Sei15} and prove the proposition by induction on the number of generators of $G$.
If $G$ is trivial, the claim is clear. For a non-trivial group $G$ we write $G = \tilde{G} \times H$ with a cyclic subgroup $H$
and apply the inductive hypothesis to the extension $M/K$ where $M := L^H$.  To that end we put
$u_M := \N_{L/M}(u)$. Then $u_M \in U_{M/K}$ and we first note that 
$(u_{M}, \psi)_{M/K}=0$ for all $\psi \in \widehat{G / H}$ because by the very definition of the pairing we have
\hbox{$(u, \chi)_{L/K} = (u_M, \chi)_{M/K}$} for all $\chi \in \hat{G}$ with $H \subseteq \ker(\chi)$.

By induction we obtain a finite unramified extension $K^{\prime}/K$, an integer $r^{\prime}$, units 
$\beta_{1}^{\prime}, \ldots , \beta_{r^{\prime}}^{\prime} \in \O_{M^{\prime}}^{\times}$ (where 
$M^{\prime}=M K^{\prime}$) and automorphisms 
$\sigma_{1}^{\prime}, \ldots , \sigma_{r^\prime}^{\prime} \in \Gal(M^{\prime}/K^{\prime})$ 
such that
\begin{equation}\label{long proof eq 1}
u_M = \N_{L/M}(u) = \prod_{i=1}^{r'} (\beta^{\prime}_{i})^{\sigma^{\prime}_{i}-1}.
\end{equation}
By applying Lemma \ref{lemma_translateunramified} to $\beta_1^\prime, \ldots, \beta_{r^\prime}^\prime$ 
and the extension $L'/M'$ we obtain an unramified extension
$M^{\prime\prime} / M^\prime$, elements $\betapp_{1}, \ldots , \betapp_{r^\prime} \in \O_{L^{\prime \prime}}^{\times}$ 
(where $L^{\prime \prime}= L^{\prime} M^{\prime \prime}$) such that
\begin{equation}\label{long proof eq 2}
\beta_{i}^{\prime} = \N_{L^{\prime \prime}/M^{\prime \prime}}(\betapp_{i}) 
\end{equation}
With  $K^{\prime \prime}/K^\prime$ denoting the unramified extension of degree $[M'' : M']$ we have the following diagram
\[
\begin{tikzcd}[column sep=small]
{}&& L^{\prime \prime} \\
& L^{\prime} \arrow[-]{ur} & M^{\prime \prime} \arrow[-]{u} \\
L \arrow[-]{ur} & M^{\prime} \arrow[-]{u} \arrow[-]{ur} & K^{\prime \prime} \arrow[-]{u} \\
M \arrow[-]{u}{} \arrow[-]{ur} & K^{\prime} \arrow[-]{u} \arrow[-]{ru}{} & \\
K \arrow[-]{u}{} \arrow[-]{ru}{}&&
\end{tikzcd}
\]
As a consequence, restriction induces  a canonical epimorphism 
$\Gal(L^{\prime \prime}/K^{\prime \prime}) \twoheadrightarrow \Gal(M^{\prime}/K^{\prime})$
and we may choose lifts $\sigmapp_{1}, \ldots,  \sigmapp_{r'} \in \Gal(L''/K'')$  of the elements
$\sigmap_{1}, \ldots,  \sigmap_{r'} \in \Gal(M'/K')$. We set
\[
u^{\prime \prime}:= u \cdot \left( \prod_{i=1}^{r^{\prime}} (\betapp_i)^{\sigmapp_i -1} \right)^{-1} \in \O_{L^{\prime\prime}}^{\times}.
\]
Then a straightforward computation using (\ref{long proof eq 1}) and (\ref{long proof eq 2}) 
shows that $u'' \in U_{L''/M''}$. We let $\tau$ denote a generator of $H$ and apply Hilbert's Theorem 90 to obtain an
element  $b^{\prime \prime} \in (L^{\prime \prime})^{\times}$ such that
 $u^{\prime \prime}= (b^{\prime \prime})^{\tau -1}$, and hence
\[
u = (b'')^{\tau - 1} \cdot  \prod_{i=1}^{r^{\prime}} (\betapp_i)^{\sigmapp_i -1}.
\]
Since we can adapt $b''$ by elements of $(M'')^\times$, this proves the proposition, provided that we can show the following claim.
\begin{claim_wn} There exists an element $a'' \in (M'')^\times$ such that $a''b'' \in \calO_{L''}^\times$. \end{claim_wn}
For the proof of the claim we first note that 
 \begin{equation}\label{long proof eq 3}
(u'', \psi)_{L''/M''}=0 \text{ for all } \psi \in \widehat{\Gal(L''/M'')}.
\end{equation}
Indeed, if we define $\tilde\psi \in \widehat{\Gal(L''/K'')}$  by $\tilde\psi |_H = \psi$ and $\tilde\psi |_{\tilde{G}} = 1$,
then Lemma \ref{lemma_auxlemmatomainlemma} below shows that 
$ (u^{\prime \prime}, \psi)_{L^{\prime \prime}/M^{\prime \prime}} = 
(u^{\prime \prime}, \tilde{\psi})_{L^{\prime \prime}/K^{\prime \prime}}$. Furthermore, 
\begin{align*}
(u^{\prime \prime}, \tilde{\psi})_{L^{\prime \prime}/K^{\prime \prime}} 
&= (u, \tilde{\psi})_{L^{\prime \prime}/K^{\prime \prime}} - \sum_{i=1}^{r^{\prime}}\left(\frac{\sigma^{\prime \prime}_{i}(\betapp_{i})}{\betapp_{i}}, \tilde{\psi} \right)_{L^{\prime \prime}/K^{\prime \prime}} \\
&= (u, \tilde{\psi})_{L/K} - \sum_{i=1}^{r^{\prime}}\left(\frac{\sigma^{\prime \prime}_{i}(\betapp_{i})}{\betapp_{i}}, \tilde{\psi} \right)_{L^{\prime \prime}/K^{\prime \prime}} \\
&= 0,
\end{align*}
where the second equality holds because $K''/K$ is unramified and the last equality follows from $(u, \tilde{\psi})_{L/K} = 0$ (by assumption) and \hbox{Lemma \ref{exp formula}}.

We are finally ready to prove the above claim. Let $\chi \in \Gal(L''/M'')$ 
be defined by $\chi(\tau) = \frac{1}{[L'':M'']} + \Ze$. We write $e_{L''/M''}$ for the ramification index of $L''/{M^{\prime \prime}}$.
By (\ref{long proof eq 3}) and the definition of the pairing 
we get 
\[
0 = (u^{\prime \prime}, \psi)_{L^{\prime \prime}/M^{\prime \prime}} = v_{M''}(b'') = \frac{1}{e_{L''/M''}}v_{L''}(b'')
\]
in  $\frac{1}{e_{L''/M''}}\Ze / \Ze$ and this implies that $e_{L''/M''}$ divides $v_{L''}(b'')$. This, in turn, guarantees the 
existence of $a''$ as in the above claim.
\end{proof}

\begin{lemma}
\label{lemma_auxlemmatomainlemma}
Let $L/K$ be a totally ramified finite abelian extension with Galois group $G$. 
Suppose that $G = \tilde{G} \times H$ with a cyclic subgroup $H$. Set $M:=L^{H}$ and $\tilde{M}:= L^{\tilde{G}} $.
For $\psi \in \widehat{H}$ we define $\tilde{\psi} \in \hat{G}$ by $\tilde{\psi}_{\mid H} = \psi$ and $\tilde{\psi}_{\mid \tilde{G}}=1$. 
Then $(w, \tilde{\psi})_{L/K} =(w, \psi)_{L/M}$  for all  $w \in U_{L/M}$ and all $\psi \in \hat{H}$.
\end{lemma}
\begin{proof}
Let $\ll \tau \rl = H$ and define $\chi \in \widehat{H}$ by $\chi(\tau)= \frac{1}{|H|}+\Ze$.  By Proposition \ref{pairing mult} it suffices to
show that $(w, \tilde{\chi})_{L/K} =(w, \chi)_{L/M}$  for all  $w \in U_{L/M}$.

Let $\beta \in L^{\times}$ such that $\frac{\tau (\beta)}{ \beta}=w$. 
Then $(w, \chi)_{L/M} = v_{M}(\beta)$ in $\Qu/\Ze$. Since $\ker(\tilde{\chi})= \tilde{G}$ we have
$K_{\tilde\chi} = \tilde{M}$ and 
\[
w_{{\tilde{\chi}}}=\N_{L / \tilde{M}}(w) = \frac{\tau(\N_{L / \tilde{M}}(\beta))}{ \N_{L / \tilde{M}}(\beta)}.
\] 
Therefore, by definition of the pairing,  $(w, \tilde{\chi})_{L/K} = v_{K}(\N_{L/ \tilde{M}}(\beta))$ and a straightforward computation with
valuations shows that $v_{K}(\N_{L/ \tilde{M}}(\beta)) = v_{M}(\beta)$.
\end{proof}

Later we will need the following definition. 

\begin{definition}
\label{def_pairing_inf}
Let $H$ be a local field and $H_{\infty}/H$ an infinite abelian extension. Let \hbox{$(u_{N})_{N} \in \varprojlim_{N}N^\times$}, where $N/H$ varies over the
finite subextensions of $H_\infty/H$, be a norm-coherent sequence with $\N_{N/H}(u_N) = 1$. Furthermore, let
$\chi$ be a character of finite order of  $\Gal(H_{\infty}/H)$. Choose $N$ such that $H_{\chi} \subset N$. Then we  set
\[
u_{\chi}:= \N_{N/H_{\chi}}(u_{N})
\]
and define a  pairing for the extension $H_{\infty}/H$ by 
\[
(u, \chi)_{H_{\infty}/H}:=(u_{\chi}, \chi)_{H_{\chi}/H}.
\]
\end{definition}
{It is easy to see that multiplicativity in both variables follows from the finite dimensional case.}

\subsection{Relative Lubin-Tate Groups and Coleman power series}\label{Rel Lubin-Tate}

In this subsection we introduce the notion of relative Lubin-Tate formal groups and 
also recall some results of the theory of Coleman power series. All results presented here can be found in \cite{dS87}.

Let $H$ be a finite extension of $\Q_{p}$, let ${\O_H}$ and ${\frp_H}$ be its valuation ring and valuation ideal, respectively.
We let $q$ denote the cardinality of the residue class field $\O_{H}/\frp_{H}$. We fix an integer $d > 0$ and let $H^{\prime}$ 
be the unramified extension of $H$ of degree $d$. We write $\varphi \in \Gal(H'/H)$ for the arithmetic Frobenius element.

We write ${\O_{H^{\prime}}}$ and ${\frp_{H^{\prime}}}$ for the  valuation ring and valuation ideal in $H'$ and fix an element $\xi \in H^\times$ with $v_H(\xi) = d$.
We set
\[
\mathcal{F}_{\xi}:= \{ f \in \O_{{H^{\prime}}}[[T]] :  
f \equiv \pi^{\prime} T (\bmod \deg 2), \; \N_{H^{\prime}/H}(\pi^{\prime})=\xi, \; f \equiv T^{q} (\bmod \frp_{{H^{\prime}}}[[T]]) \}
\]
and recall from \cite[Ch.~I, Thm~1.3]{dS87} that for each $f \in \mathcal{F}_\xi$ there exists a unique one-dimensional commutative 
formal group law $F_{f} \in \O_{{H^{\prime}}}[[X,Y]]$  satisfying \hbox{$F_f^\varphi \circ f = f \circ F_{f}$}. 
We call $F_{f}$  a \emph{relative} Lubin Tate group (relative to the extension $H^{\prime} / H$). The case $d=1$ corresponds to classical
Lubin-Tate formal groups. 

Let $f, g \in \calF_{\xi}$ with $f=\pi_{1} T + \ldots$ and $g=\pi_{2}T+ \ldots$. 
For an element $a \in \O_{H^{\prime}}$ such that $a^{\varphi -1} = \pi_{2}/\pi_{1}$ there is a 
power series $[a]_{f,g} \in \O_{H^{\prime}}[[T]]$ uniquely determined by the properties of  \cite[Ch. I, Prop. 1.5]{dS87}.
If $f=g$ we write $[a]_f$ in place of $[a]_{f,f}$ and note that the map  $\O_{{H}} \lra \End(F_{f}), a \mapsto [a]_f$, is an injective ring homomorphism.

Let $H^{c}$ be  the algebraic closure of $H$ and write $\frp_{H^{c}}$ for its valuation ideal. Then we get an $\O_{{H}}$-module structure on 
$\frp_{H^{c}}$ by setting
\[
x +_{f} y := F_{f}(x,y) \text{ { and } } a \cdot x := [a]_{f}(x)
\]
for $x,y \in \frp_{H^{c}}$ and $a \in \O_{{H}}$.

For an integer $n \ge 0$ and $f \in \calF_\xi$ we set $f^{(n)} := \varphi^{n-1}(f) \circ \cdots \circ \varphi(f) \circ f$.
Let $\pi$ be a prime element of $\O_H$. We define
\[
W_{f}^{n}:= \{ \omega \in  \frp_{H^{c}} \mid [\pi^{n}]_{f}(\omega)=0 \}
= \{ \omega \in  \frp_{H^{c}} \mid f^{(n)}(\omega) = 0 \}
\]
and call  $W_{f}^{n}$ the \emph{group of division points of level $n$} of $F_{f}$. 
We also set $\widetilde{W}^{n}_{f}:= W^{n}_{f} \setminus W^{n-1}_{f}$ and $W_{f}= \bigcup_{n} W^{n}_{f}$. So
 $W_{f}$ is the subgroup of all torsion points of $F_f$.

We fix $f \in \calF_\xi$ and set  $H_n^\prime := H^{\prime}(W^{n+1}_{f})$. Note that  $H_n^\prime$ does not 
depend on the choice of $f \in \mathcal{F}_{\xi}$. 
It is a totally ramified finite abelian extension of $H^{\prime}$ of degree $(q-1)q^{n}$.
Any $\omega_{n+1} \in \widetilde{W}^{n+1}_{f}$ generates $H'_n$ over $H^{\prime}$ 
and, in addition,  is a prime element in $\O_{H'_n}$. For the ring of integers in $H_n'$ we obtain $\O_{H_n'} = \O_{{H^{\prime}}}[\omega_{n+1}]$ for each
$\omega_{n+1} \in \widetilde{W}^{n+1}_{f}$.
In this context local class field theory is very explicit. The reciprocity map $\rec_H$
induces a group isomorphism $\left( \O_{H} / \frp^{n+1}\right)^{\times} \xrightarrow{\simeq} \Gal(H_n^{\prime}/H^{\prime}), 
u  \mapsto \sigma_{u} $ with $\sigma_{u}(\omega)=[u^{-1}]_{f}(\omega)$ for all $\omega \in W_{f}^{n+1}$, see {\cite[Ch. I, Prop.~1.8]{dS87}}.

In the following we introduce the \emph{Coleman norm operator} and recall some of its properties. Let $R= \O_{{H^{\prime}}}[[T]]$ be the ring of 
power series with coefficients in $\O_{{H^{\prime}}}$. By {\cite[Ch.~I, Prop.~2.1]{dS87} there exists a unique multiplicative operator
$\NC =  \NC_{f} \colon R \rightarrow R$ such that
\[
\NC h \circ f = \prod_{\omega \in W^{1}_{f}} h(T +_f \omega) 
\]
for all  $h \in R$.

\begin{proposition}\label{prop_propertiescolemannormop} (\cite[Ch.~I, Prop.~2.1]{dS87})
The Coleman norm operator has the following properties:
\begin{enumerate}[label=\alph*)]
\item $\NC h \equiv h^{\varphi} (\bmod \frp_{{H^{\prime}}})$.
\item $\NC_{\varphi(f)}= \varphi \circ \NC_{f} \circ \varphi^{-1}$.
\item Let $\NC_{f}^{(i)} := \NC_{\varphi^{i-1}(f)} \circ \cdots \circ \NC_{\varphi(f)} \circ \NC_{f}$. Then 
\[
\left( (\NC^{(i)}_{f} h) \circ f^{(i)} \right)(T) = \prod_{\omega \in W_{f}^{i}} h(T +_f \omega).
\]
\item If $h \in R$, $h \equiv 1 \bmod (\frp_{{H^{\prime}}})^ i$ for $i \geq 1$, then $\NC_f h \equiv 1 \left(\bmod (\frp_{{H^{\prime}}})^{ i+1}\right)$.
\end{enumerate}
\end{proposition}

{
Let $\alpha = (\alpha_{n}) \in \varprojlim_{} (H_n^{\prime})^\times$ be a norm-coherent sequence. 
We fix  $\omega_{i} \in \widetilde{W}^{i}_{\varphi^{-i}(f)} $ such that  $(\varphi^{-i}f)(\omega_{i})= \omega_{i-1}$.
There is a unique integer $\nu(\alpha)$ such that $\alpha_{n} \O_{H_{n}^{\prime}}=\frp_{H_{n}^{\prime}}^{\nu(\alpha)}$ for all $n \ge 0$.
By \cite[Ch.~I, Thm~2.2]{dS87}  there exists a unique power series $\Col_{\alpha} \in T^{\nu(\alpha)} \cdot \O_{H^{\prime}}[[T]]^{\times}$ such that
\begin{equation}\label{def eq Col}
(\varphi^{-(i+1)} \Col_{\alpha})(\omega_{i+1})= \alpha_{i}
\end{equation}
 for all  $i \geq 0$.
 }
The power series $\Col_{\alpha}$ is called the \emph{Coleman power series} associated to $\alpha$.
We recall that by \cite[Ch.~I, Cor.~2.3 (i)]{dS87} Coleman power series are multiplicative in $\alpha$, i.e.
\[
\Col_{\alpha \alpha^{\prime}}= \Col_{\alpha} \cdot \Col_{\alpha^{\prime}}
\]
for norm-coherent sequences $\alpha, \alpha' \in \varprojlim_{n} (H_n^{\prime})^\times$.

\begin{remark}
If we fix $\omega_{n} \in \widetilde{W}^{n}_{\varphi^{-n}(f)}$ such that $(\varphi^{-n}f)(\omega_{n})=\omega_{n-1}$ for $1 \leq n < \infty$, 
then we call $\omega = (\omega_{n})_{n \geq 0}$ a \emph{generator of the Tate module of $F_{f}$}. Note 
that each $\omega_{n}$ is a division point on $F_{\varphi^{-n}(f)} = F_f^{\varphi^{-n}}$. 
\end{remark}

We set $H^{\prime}_{\infty}:= \bigcup_{n}H_n^{\prime}$ and 
let $u=(u_{n}) \in \varprojlim \O^{\times}_{H_n^{\prime}}$ be a norm-coherent sequence of units. 
For later reference we recall the following lemma.

\begin{lemma}\label{lemma_ColpowerseriesDS87p21}
For $\sigma \in \mathcal{G} := \Gal(H_\infty'/H)$, there exits a unique isomorphism $h: F_{f} \simeq F_{\sigma(f)}$ such 
that $h(\omega)= \sigma(\omega)$ for all $\omega \in W_{f}$. This $h$ is of the form 
$[\kappa(\sigma)]_{f, \sigma(f)}$ for a unique $\kappa(\sigma) \in \O^{\times}_{\locKp}$, 
and $\kappa(\sigma)^{\varphi -1}=f^{\prime}(0)^{\sigma -1}$. The map $\kappa: \mathcal{G} \rightarrow \O^{\times}_{\locKp}$ is a $1$-cocycle, i.e., 
$\kappa(\tau \sigma)= \kappa(\sigma)^{\tau} \cdot \kappa(\tau)$ for all $\sigma, \tau \in \mathcal{G}$,
and $\Col_{u}$ and $\Col_{\sigma(u)}$ are related by
\[
\Col_{\sigma(u)}= \Col_{u}^{\sigma} \circ [\kappa(\sigma)]_{f, \sigma(f)}.
\]
\end{lemma}
\begin{proof}
This is a generalization of {\cite[Ch. I, Cor. 2.3]{dS87}} or (15) on page 21 in Chapter I.3 of \cite{dS87}. 
{The arguments used to prove this assertion in \cite{dS87} do not need the assumption 
that $p$ splits in $k/\Qu$ which is assumed in \cite[Ch.~I.3]{dS87}.}
\end{proof}

\subsection{Auxiliary results}\label{aux section}
We continue to use the notation introduced in Section \ref{Rel Lubin-Tate}.

We fix $f \in \calF_\xi$.
For this subsection we fix an integer $m \in \Ze_{>0}$ but usually suppress $m$ in our notations. Let $\pi = \pi_H$ be a
uniformizing element and define
\[
R = R_{H'} := \frac{\O_{H'}[[T]]}{\left( [\pi^{m+1}]_f \right)} = \frac{\O_{H'}[[T]]}{\left( f^{(m+1)} \right)}.
\]
We write $\iota := \iota_{H'}$ for the injective ring homomorphism 
\begin{eqnarray*}
  \iota_{H'} \colon R &\lra& \O_{H'} \oplus \bigoplus_{l=0}^m \O_{H_l'}, \\
\bar{g} &\mapsto& \left( g(0), \left( (\varphi^{-(l+1)}g)(\omega_{l+1}) \right)_{l = 0,\ldots,m}   \right).
\end{eqnarray*}
We let $r=r_{H'} \colon R \lra \O_{H_m'}$ be the composite of $\iota$ and the projection to the last component.

If $F/H'$ is a finite unramified extension, we set $R_F := R \tensor_{\O_{H'}}\O_F$ and $F_l := FH_l'$ and note that $\O_{F_l} = \O_F\O_{H_l'}$.
Let $\iota_F \colon R_F \lra \O_{F} \oplus \bigoplus_{l=0}^m \O_{F_l}$ and $r_F \colon R_F \lra \O_{F_m}$ denote the base change
of $\iota$ and $r$ along $\O_F$ over $\O_{H'}$. Since $\O_{F_m} = \O_F[\omega_{m+1}]$ the ring homomorphism $r_F$ is surjective. In addition,
as it is actually a homomorphism of local rings, we conclude that $r_F \colon R_F^\times \lra \O_{F_m}^\times$ is surjective as well.

The Galois group $\Gal(F_\infty/F)$ naturally acts on $\O_{F} \oplus \bigoplus_{l=0}^m \O_{F_l}$. The following lemma shows that we
can transport this action to $R_F$ via $\iota_F$.

\begin{lemma}\label{Galois action}
 For $\bar{g} \in R_F$ and $\sigma \in \Gal(F_\infty/F)$ the element $\sigma( \iota_F(\bar{g}) )$ is contained in the image $\iota_F(R_F)$.
\end{lemma}

\begin{proof}
We identify $\Gal(F_\infty/F)$ and $\Gal(H'_\infty/H')$ via restriction and let
$u \in \O_H^\times$ such that $\sigma = \rec_H(u^{-1})$. Recall that $\omega_{l+1}$ is a torsion point for $F_{\varphi^{-(l+1)}(f)} = F_f^{\varphi^{-(l+1)}}$
and hence, 
\begin{equation}\label{rec eq}
\sigma(\omega_{l+1}) = \rec_H(u^{-1})(\omega_{l+1}) = [u]_f^{\varphi^{-(l+1)}}(\omega_{l+1})
\end{equation}
by \cite[Ch. I, Prop.~1.8]{dS87}.

We thus obtain
\begin{eqnarray*}
  \sigma(\iota_F(\bar{g})) &=& \left( \sigma(g(0)), \left( \sigma( (\varphi^{-(l+1)}g)(\omega_{l+1}) ) \right)_{l = 0,\ldots,m}   \right) \\
                                         &=& \left( g(0), \left(  (\varphi^{-(l+1)}g)([u]_f^{\varphi^{-(l+1)}}(\omega_{l+1}))  \right)_{l = 0,\ldots,m}   \right) \\
                                         &=& \left( g(0), \left(  (\varphi^{-(l+1)}(g \circ [u]_f))(\omega_{l+1})  \right)_{l = 0,\ldots,m}   \right) \\
                                         &=& \iota_F \left( \overline{g \circ [u]_f} \right),
\end{eqnarray*}
where we use (\ref{rec eq}) for the second equality.
\end{proof}

For $\bar{g} \in R_F$ we define $\N_{R_F/\O_F}(\bar{g})$ to be the norm of the $\O_F$-linear endomorphism of $R_F$ given by multiplication by $\bar{g}$.

\begin{lemma}\label{norm on R_F}
Let $\bar{g} \in R_F$. Then:\\
\begin{enumerate}[label=\alph*)]
\item $$\N_{R_F/\O_F}(\bar{g}) = g(0) \prod_{l=0}^m \left(  \N_{F_l/F}((\varphi^{-(l+1)}g) (\omega_{l+1})    \right)^{\varphi^{l+1}}.$$
\item $$\N_{R_F/\O_F}(\bar{g}) = \prod_{\omega \in W_f^{m+1}} g(\omega).$$
\end{enumerate}

\end{lemma}

\begin{proof}
 For the proof of a) we first define a modified $\O_F$-module structure on $\O_{F} \oplus \bigoplus_{l=0}^m \O_{F_l}$ by
\[
a * (\beta_{-1}, \beta_0, \ldots, \beta_m) := (a\beta_{-1}, \varphi^{-1}(a)\beta_0, \ldots, \varphi^{-(m+1)}(a)\beta_m)
\]
for $a, \beta_{-1} \in \O_F$ and $\beta_l \in F_{l}$, $l \in \{0,\ldots,m\}$. With respect to this $\O_F$-module structure
$\iota_F$ is a homomorphism of $\O_F$-modules.

In the same way we define a new $\O_F$-module structure on each of the fields $F_{l}$ for $l\in~\{0,\ldots,m\}$. Explicitly, $a*\alpha := \varphi^{-(l+1)}(a)\alpha$ for $a \in \O_F$ and $\alpha \in \O_{F_{l}}$.

We fix $l \in \{0, \ldots, m\}$ and let $\alpha_1, \ldots, \alpha_s$ with $s := [F_l:F]$ denote an $\O_F$-basis of $\O_{F_l}$ with respect
to the usual $\O_F$-module structure given by multiplication. Then, for $\beta \in \O_{F_l}$ and $i \in \{1,\ldots,s\}$, there exist elements
$a_{ij} \in \O_F$ such that
\begin{equation}\label{norm eq}
  \beta\alpha_i = \sum_{j=1}^s a_{ij}\alpha_j =  \sum_{j=1}^s \varphi^{l+1}(a_{ij}) * \alpha_j.
\end{equation}
Multiplication by $\beta$ is $\O_F$-linear with respect to both $\O_F$-module structures on $\O_{F_l}$ and we write
$\N_{F_l/F}$, respectively $\N_{F_l/F}^*$, for the induced norm maps. Then (\ref{norm eq}) implies
\[
\N_{F_l/F} (\beta)^{\varphi^{l+1}} = \det\left(\left( a_{ij} \right)_{i,j = 1,\ldots,s}\right)^{\varphi^{l+1}} = \N_{F_l/F}^* (\beta).
\]
Hence  part a) of  the lemma follows from 
\[
\N_{R_F/\O_F}(\bar{g}) = g(0) \prod_{l=0}^m \left(  \N_{F_l/F}^*((\varphi^{-(l+1)}g) (\omega_{l+1}) )   \right),
\]
which in turn is immediate from $\O_F$-linearity of $\iota_F$ with respect to the modified $\O_F$-module structure.

In order to prove b) we fix an element $\tau \in \Gal(F_\infty/H)$ such that $\tau|_{F} = \varphi$. Then we obtain from a)
\[
\N_{R_F/\O_F}(\bar{g}) = g(0) \prod_{l=0}^m \left(  \N_{F_l/F}((\varphi^{-(l+1)}g) (\omega_{l+1})    \right)^{\tau^{l+1}}
                                       = g(0) \prod_{l=0}^m \N_{F_l/F}( g (\omega_{l+1}^{\tau^{l+1}})).
\]
Since $\Gal(F_l/F)$ acts simply transitive on $\widetilde{W_f^{l+1}}$ the result easily follows.
\end{proof}

\subsection{Seiriki's theorem on the constant term of a Coleman power series}
We let $F_f$ be a Lubin-Tate formal group relative to the unramified extension $H'/H$ and resume the 
notations of Section \ref{Rel Lubin-Tate}. Recall that $H_n' = H'(W_f^{n+1})$ for $n \ge 0$. We also set
$H_\infty' := \cup_{n \ge 0}H_n'$. If $\chi$ is a character of finite order of $\Gal(H_\infty' / H')$ we set $H_\chi' := (H_\infty')^{\ker(\chi)}$.
For a norm-coherent sequence \hbox{$u = (u_n)_{n \ge 0} \in \projlim{n} \O_{H_n'}^\times$} we set $u_{H'} := \N_{H_n'/H'}(u_n)$ for any
$n \ge 0$.

\begin{theorem}\label{Seiriki's theorem}
  Let $\chi$ be a character of finite order of $\Gal(H_\infty' / H')$. Let $u = (u_n)_{n \ge 0} \in \projlim{n} \O_{H_n'}^\times$
be a norm-coherent sequence with $u_{H'} = 1$. In addition, we assume $\Col_u(0) \in \O_H^\times$. Then
\begin{equation}\label{Sei main eq}
  (u, \chi)_{H_\infty'/H'} = - \chi(\rec_H(\Col_u(0))).
\end{equation}
\end{theorem}

\begin{remark}{In the case $H=H'$ this is essentially Corollary 2.8 in \cite{Sei15}.
In addition, there is a minus missing in the statement of this corollary.} 
\end{remark}

\begin{remark}\label{independence remark}
\begin{enumerate}[label=\alph*)]
\item The right hand side does not depend on the choice $\omega = (\omega_n)_{n\ge 1}$ of a generator of the Tate module. Indeed, if
$\omega' = (\omega_n')_{n\ge 1}$ is another such generator, then there is a unique $\sigma \in \Gal(H_\infty'/H')$ such that
$\sigma(\omega_n) = \omega_n'$ for all $n \ge 1$. By local class field theory there exists a unique $v \in \O_H^\times$ with $\rec_H(v) = \sigma$.
Then $\omega_n' = \sigma(\omega_n) =  [v^{-1}]_f^{\varphi^{-n}}(\omega_n)$ (since $\omega_n$ is a torsion point of $F_{\varphi^{-n}(f)}$).

If $\Col_u'$ denotes the Coleman power series with respect to $\omega'$, then
$\Col_u' = \Col_u \circ [v]_f$ and thus $\Col_u'(0) = \Col_u(0)$.

\item Without loss of generality we may assume that the sequence $\omega = (\omega_n)_{n \ge 1}$ is norm-coherent. To show
this we apply \cite[Lemma~4.1]{Bl04} and proceed as follows: We fix a norm-coherent sequence $\beta = (\beta_n)_{n \ge 1}$
of prime elements of $H_n'$ and let $\Col_\beta \in T\O_{H'}[[T]]$ be the associated Coleman power series. Let $\Col_\beta^{-1} \in T\O_{H'}[[T]]$
be such that $\Col_\beta \circ \Col_\beta^{-1}  = T$. 
If we set $f' := \Col_\beta^\varphi \circ f \circ \Col_\beta^{-1}$, then $f' \in  \mathcal{F}_\xi$ and the proof of \cite[Lemma~4.1 b)]{Bl04} shows that
$\beta$ is a generator of the Tate module for $F_{f'}$. With respect to $F_{f'}$ and $\beta$ the Coleman power series associated
to $u$ is equal to $\Col_u \circ \Col_\beta^{-1}$ and $(\Col_u \circ \Col_\beta^{-1})(0) = \Col_u(0)$, so that we may prove the theorem
for $f$ replaced by $f'$ and $\omega$ replaced by $\beta$.
\end{enumerate}
\end{remark}

\begin{proof}
We fix $m > 0$ and note that it suffices to prove the theorem for an arbitrary character $\chi$ of $\Gal(H_m'/H')$. We write
$\Gal(H_m'/H')$ as a direct product of cyclic subgroups,
\[
 \Gal(H_m'/H') = G_1 \times \ldots \times G_s,
\]
with $s \ge 1$ and for each $i \in \{ 1, \ldots, s\}$ we set $U_i := \prod_{j \ne i}G_j$ with the subscript $j$ ranging over $\{1,\ldots,s\}$. 
For each $j$ we fix a generator
$\sigma_j$ of $G_j$ and define a character $\chi_j \in \hat{G}_j$ by $\chi_j(\sigma_j) = \frac{1}{|G_j|} + \Ze$. Then the characters
$\chi_1, \ldots, \chi_s$ generate the group of characters of $\Gal(H_m'/H')$.

{\bf Claim 1: } For $j = 1, \ldots, s$ there exist units $b_{u, \chi_j} \in \O_H^\times$ such that
$(u_{\chi_j}, \chi_j)_{H_{\chi_j}'/H'} = \chi_j(\rec_H(b_{u, \chi_j}^{-1}))$ and $\rec_H(b_{u, \chi_j}) \in G_j$.

For the proof of Claim 1 we fix $a_j \in \Ze_{>0}$ such that $(u_{\chi_j}, \chi_j)_{H_{\chi_j}'/H'} =\frac{a_j}{|G_j|} + \Ze$. Since
$\rec_H$ induces an isomorphism 
$\left( \O_H / \frp_H^{m+1} \right)^\times \simeq \Gal(H_m'/H')$ 
there exists $b_{u, \chi_j} \in \O_H^\times$
such that $\rec_H(b_{u, \chi_j}^{-1}) = \sigma_j^{a_j}$. Claim 1 is now immediate from $\chi_j(\sigma_j) = \frac{1}{|G_j|}+\Ze$.

By Remark \ref{independence remark} we may, without loss of generality, assume that the generator $\omega = (\omega_n)_{n \ge 1}$ of
the Tate module is norm coherent. We set $b_u := \prod_{j=1}^s b_{u, \chi_j} \in \O_H^\times$ and define
\[
u_n' := \frac{\rec_H(b_u^{-1})(\omega_{n+1})}{\omega_{n+1}}.
\]
Then $u' := (u_n')_{n \ge 0}$ is a norm coherent sequence of units in $\varprojlim_{n} \O_{H_n'}^\times$.

{\bf Claim 2: } $\Col_{u'}(T) = \frac{[b_u]_f(T)}{T}$, and thus $ \Col_{u'}(0) = b_u$. 

For the proof of Claim 2 we recall that $\omega_{n+1}$ is a torsion point of $F_{\varphi^{-(n+1)}(f)} = F_f^{\varphi^{-(n+1)}}$.
We have $\rec_H(b_u^{-1})(\omega_{n+1}) = [b_u]_{\varphi^{-(n+1)}(f)}(\omega_{n+1}) = [b_u]_{f}^{\varphi^{-(n+1)}}(\omega_{n+1})$.
Put $g :=  \frac{[b_u]_f(T)}{T}$ and observe that
\[
(\varphi^{-(n+1)}(g))(\omega_{n+1}) = \frac{[b_u]_{f}^{\varphi^{-(n+1)}}(\omega_{n+1})}{\omega_{n+1}}
= \frac{\rec_H(b_u^{-1})(\omega_{n+1})}{\omega_{n+1}} = u_n',
\]
so that $g = \frac{[b_u]_f(T)}{T}$ satisfies the defining equality (\ref{def eq Col}) for all $n \ge 0$.

We now set $u'' := u/u'$ and obtain from Claim 2 that $\Col_{u''}(0) \cdot b_u = \Col_u(0)$. In particular, 
since we assume that $\Col_{u}(0)$ is contained in $\O_H^\times$, it follows that $\Col_{u''}(0) \in \O_H^\times$.

The right hand side of (\ref{Sei main eq}) is obviously multiplicative in the character, for the left hand side this
is shown in Proposition \ref{pairing mult}. Thus, it suffices to prove the theorem for each of the characters $\chi_i$,
$i = 1, \ldots, s$. 

We henceforth fix $i \in \{1, \ldots, s\}$. We observe that by Claim 1 
$\chi_i(\rec_H(b_u)) = \chi_i(\rec_H(b_{u, \chi_i}))$ and $(u, \chi_i)_{H_\infty'/H'} =  - \chi_i(\rec_H(b_{u, \chi_i}))$. It is now easy to see
that for $\chi = \chi_i$ the equality (\ref{Sei main eq}) is equivalent to
$ \chi(\rec_H(\Col_{u''}(0))) = 0$.
Since $\chi_i$ is a character of $\Gal(H_m'/H')$ and $\rec_H$ induces an isomorphism \hbox{$\left( \O_H/\frp_H^{m+1} \right)^\times \simeq \Gal(H_m'/H')$} 
it thus suffices to show that
\begin{equation}\label{Sei main eq i}
\Col_{u''}(0) \equiv 1 \pmod{\frp_H^{m+1}}.
\end{equation}

{\bf Claim 3:} $(u_m'', \psi)_{H_m'/H'} = 0$ for all characters $\psi$ of $\Gal(H_m'/H')$.

Because of the multiplicativity result of Proposition \ref{pairing mult} it suffices to show that $(u_m'', \chi_j)_{H_m'/H'} = 0$ for $j = 1, \ldots,s$.
We fix $j$ and compute
\[
(u_m'', \chi_j)_{H_m'/H'} = (u_m, \chi_j)_{H_m'/H'} - (u_m', \chi_j)_{H_m'/H'} = \frac{a_j}{|G_j|} - (u_m', \chi_j)_{H_m'/H'}
\]
with $a_j \in \Ze_{>0}$ as in the proof of Claim 1. 
Hence it suffices to show that  the equality  $(u_m', \chi_j)_{H_m'/H'} = \frac{a_j}{|G_j|} + \Ze$ is valid.
We set $\eta_{m+1} := \N_{H_m'/H_{\chi_j}'}(\omega_{m+1})$ and note that
$H_{\chi_j}' = (H_m')^{U_j}$. By  Claim 1 and its proof this implies
\begin{eqnarray*}
 \N_{H_m'/H_{\chi_j}'}(u_{m}') &=& \frac{\rec_H(b_u^{-1})(\eta_{m+1})}{\eta_{m+1}} =
 \frac{\left( \prod_{l=1}^s\rec_H(b_{u,\chi_l}^{-1}) \right) (\eta_{m+1})}{\eta_{m+1}}\\
 &=& \frac{\rec_H(b_{u,\chi_j}^{-1})(\eta_{m+1})}{\eta_{m+1}}
= \frac{\sigma_j^{a_j}(\eta_{m+1})}{\eta_{m+1}}.
\end{eqnarray*}
By definition of the pairing and Lemma \ref{exp formula} we conclude further
\begin{eqnarray*}
  \left( u_m', \chi_j \right)_{H_m'/H'} &=& \left( \frac{\sigma_j^{a_j}(\eta_{m+1})}{\eta_{m+1}}, \chi_j \right)_{H_{\chi_j}'/H'} \\
&=& v_{H_{\chi_j}'}(\eta_{m+1}) \chi_j(\sigma_j^{a_j}) =  \chi_j(\sigma_j^{a_j}) = \frac{a_j}{|G_j|} + \Ze,
\end{eqnarray*}
as required.

The following two claims now conclude the proof of (\ref{Sei main eq i}), and hence also the proof of Theorem \ref{Seiriki's theorem}.
We will use the notation and results of Section \ref{aux section}.

{\bf Claim 4:} There exists a finite unramified extension $F/H'$ and an element $\widetilde{u_m''} \in R_F^\times$ such that 
$r_F(\widetilde{u_m''}) = u_m''$ and $\N_{R_F/\O_F}(\widetilde{u_m''}) = 1$.

{\bf Claim 5:} For any $\bar{x} \in R_F^\times$ with $r_F(\bar{x}) = 1$ one has $ \N_{R_F/\O_F}(\bar x) \equiv 1 \pmod{\frp_F^{m+1}}$.

Indeed, by Lemma \ref{norm on R_F}, the defining equality (\ref{def eq Col}) 
for Coleman power series and the fact that $\N_{F_l/F}(u_l'') = 1$ for all $l$ we have 
\begin{eqnarray*}
\N_{R_F/\O_F}(\overline{\Col_{u''}}) &=& \Col_{u''}(0) \prod_{l=0}^{m} \left( \N_{F_l/F}\left( (\varphi^{-(l+1)}\Col_{u''})(\omega_{l+1}) \right)\right)^{\varphi^{l+1}}\\
                                                       &=& \Col_{u''}(0) \prod_{l=0}^{m} \left( \N_{F_l/F}(u_l'') \right)^{\varphi^{l+1}}\\                                                       
                                                       &=&  \Col_{u''}(0).
\end{eqnarray*}
Since $r_F(\overline{\Col_{u''}}) = u_m'' = r_F(\widetilde{u_m''})$ we thus conclude from Claims 4 and 5  
that \hbox{$ \Col_{u''}(0) \equiv 1 \pmod{\frp_F^{m+1}}$}. Since $  \Col_{u''}(0) \in \O_H$ and $\frp_F^{m+1} \cap \O_H = \frp_H^{m+1}$ (since $F/H$ is
unramified) the equality (\ref{Sei main eq i}) follows.

For the proof of Claim 4 we first note that by Claim 3 all assumptions of Proposition \ref{main prop} for $u_m''$ and $H_m'/H'$ are satisfied. 
Hence we  conclude that there exists a finite unramified extension $F/H'$, an integer $r > 0$, units $\beta_1, \ldots, \beta_r \in \O_{F_m}^\times$ and
automorphisms $\sigma_1, \ldots, \sigma_r \in \Gal(F_m/F)$ such that
\[
u_m^{\prime \prime} = \prod_{j=1}^r \frac{\sigma_j(\beta_j)}{\beta_j}.
\]
Since we know from Section \ref{aux section} that $r_F \colon R_F^\times \lra \O_F^\times$ is surjective, we can choose 
elements $\widetilde{\beta_j} \in R_F^\times$ such that $r_F(\widetilde{\beta_j}) =\beta_j$. 
Recall also from Lemma \ref{Galois action} that we have a natural action of $\Gal(F_m/F)$ on
$R_F$ and set
\[
 \widetilde{u_m^{\prime \prime}} := \prod_{j=1}^r \frac{\sigma_j(\widetilde{\beta_j)}}{\widetilde{\beta_j}} \in R_F^\times.
\]
So  $\widetilde{u_m^{\prime \prime}}$ is a unit in $R_F$ which by construction and Lemma \ref{norm on R_F} a) satisfies 
$\N_{R_F/\OF}(\widetilde{u_m''}) = 1$.

It finally remains to prove Claim 5. If $r_F(\bar x) = 1$ for a power series $x \in \O_F[[T]]$, 
then it is straightforward to see that $x(\omega) = 1$ for all
torsion points $\omega \in \widetilde{W_f^{m+1}}$. We set $h := \frac{f^{(m+1)}}{f^{(m)}} =  \frac{(\varphi^mf)(f^{(m)})}{f^{(m)}}$.
Then $h(T) = \tilde{\pi}+ h_1(f^{(m)})$ with a power series $h_1 \in T \O_{H'}[[T]]$ and  $\tilde\pi := \varphi^m(\pi')$, a uniformizing element in $H'$. 
The set of zeroes of $h$ is given by $\widetilde{W_f^{m+1}}$, so that a straightforward application of the Weierstrass preparation theorem
shows that $h$ divides $x-1$. We write $x = 1 + hg$ with a power series $g \in \O_F[[T]]$. 

By part a) of Lemma \ref{norm on R_F} we obtain
\begin{eqnarray*}
\N_{R_F/\O_F}(\bar{x}) &=& x(0) \prod_{l=0}^m \N_{F_l/F}\left( (\varphi^{-(l+1)}x)(\omega_{l+1}) \right)^{\varphi^{l+1}} \\
                                      &=&  x(0) \prod_{l=0}^{m-1} \N_{F_l/F}\left( (\varphi^{-(l+1)}x)(\omega_{l+1}) \right)^{\varphi^{l+1}}
\end{eqnarray*}
where the second equality holds because of $(\varphi^{-(m+1)}x)(\omega_{m+1}) = r_F(\bar{x}) = 1$. As in the proof of part b) of Lemma  \ref{norm on R_F} 
we derive
\[
\N_{R_F/\O_F}(\bar{x}) = \prod_{\omega \in W_f^m} x(\omega) = \prod_{\omega \in W_f^m} (1 + g(\omega)h(\omega)).
\]
Since  $h(T) = \tilde{\pi}+ h_1(f^{(m)})$ and $f^{(m)}(\omega) = 0$ for all $\omega \in W_f^m$ we further deduce 
\[
\N_{R_F/\O_F}(\bar{x}) = \prod_{\omega \in W_f^m} (1 + g(\omega)\tilde{\pi}).
\]
Set $j(T):= 1+ g(T) \tilde{\pi} \in \O_{\locF}[[T]]$. We note that $f \in \calF_{\xi_F}$ with $\xi_F := \xi^{[F:H']}$ and with respect to the unramified extension
$F/H$, so that the formal group $F_f$ can also be considered as a Lubin-Tate extension relative to $F/H$.  By 
Proposition \ref{prop_propertiescolemannormop} c) we therefore obtain
\[
\left( (\N^{(m)}_{f}j) \circ f^{(m)} \right) (T) = \prod_{\omega \in W^{m}_{f}} j(T +_{f} \omega).
\]
As a consequence
\[
\N^{(m)}_{f}(j)(0) = \prod_{\omega \in W^{m}_{f}} j(\omega) = \N_{R_{F}/ \O_{F}}(\bar{x}).
\]
Moreover, we have $j \equiv 1 \bmod \frp_{F}$. Applying Proposition \ref{prop_propertiescolemannormop} d) inductively we obtain
$\N_{f}^{(m)}(j) \equiv 1 \bmod \frp^{m+1}_{F}$ and hence $\N_{R_{\locF}/\O_{\locF}}(\bar{x}) \equiv 1 \bmod \frp^{m+1}_{F}$.
\end{proof}

For the proof of Theorem \ref{thm_main} we will need a variant of Theorem \ref{Seiriki's theorem}. Let 
\hbox{$u=(u_{n})_{n \geq 0} \in \varprojlim_{n} \O^{\times}_{H_n'}$} be a norm-coherent sequence.
The proof of \cite[Lemma~4.2]{Bl04} shows that  $\N_{H_n' / H}(u_{n})=1$ for all $n \ge 0$. 
We fix a set of representatives $\{\tau_1, \ldots, \tau_d\}$ of $\Gal(H_\infty'/H)$ modulo  $\Gal(H_\infty'/H')$
and define for all $n \geq 0$
\begin{equation}\label{w def}
\w_{n}:= \prod_{i=1}^{d} \tau_{i}(u_{n}).
\end{equation}
Note that $\w_n$ depends on the choice of the set $\{\tau_1, \ldots, \tau_d\}$, however, we will suppress this dependency in our notation.

\begin{lemma}
\label{lemma_propertiesvandw}
For the elements $\w_n$ constructed above, we obtain
\begin{enumerate}[label=\alph*)]
\item $\N_{H_m' / H_n'}(\w_{m})=\w_{n}$ for $m \geq n \ge 0$.
\item $\N_{H_n' / H'}(\w_{n})=1$ for all $n\geq 0$.
\item $\Col_{\w}(0) = \N_{\locKp / \locL}(\Col_{u}(0))$
\end{enumerate}
\end{lemma}
\begin{proof}
The proofs of a) and b) are immediate from the definitions. 
Since \[
\varphi^{-(j+1)} \left( \prod_{i=1}^{d} \Col_{\tau_{i}(u)}\right)(\omega_{j+1})= \prod_{i=1}^{d}\tau_{i}(u_{j})=\w_{j}
\]
for all $j \ge 0$
we know by (\ref{def eq Col}) that $\prod_{i=1}^{d} \Col_{\tau_{i}(u)}= \Col_{\w}$. By Lemma \ref{lemma_ColpowerseriesDS87p21} we
have \hbox{$\Col_{\tau_{i}(u)} = \left( \Col_{u}^{\tau_i} \right) \circ [\kappa(\tau_{i})]_{f, \tau_i(f)}$}, and as a consequence 
$ \Col_{\tau_{i}(u)}(0) = \tau_{i}(\Col_{u}(0))$  and so the result of c) obviously follows.
\end{proof}

\begin{corollary}\label{cor_sei}
Let $u=(u_{n})_{n \geq 0} \in \varprojlim_{n} \O^{\times}_{H_n'}$ be a norm-coherent sequence of units.
Then we have for each character $\chi$ of $\Gal(H_\infty'/H)$ of finite order
\[
(u, \chi)_{H_\infty'/H} = -\chi(\rec_H(\N_{H'/H}(\Col_u(0)))).
\]
\end{corollary}

\begin{proof}
We set $H_\chi := (H_\infty')^{\ker(\chi)}$ and $H_\chi' :=  (H_\infty')^{\ker(\chi) \cap \Gal(H_\infty'/H')}$. This is summarized in the following diagram of fields.

\[
\begin{tikzcd}[column sep=small, row sep=small]
{} &  & H_{\infty}^{\prime} \arrow[d, no head] \arrow[ddd, "\ker(\chi)", no head, bend left] \\
 &  & H_{\chi}^{\prime} \arrow[dd, no head] \\
H^{\prime} \arrow[dd, "\left\langle \varphi^{t} \right\rangle", no head] \arrow[rru, no head] &  &  \\
 &  & H_{\chi} \arrow[lldd, "\left\langle \sigma\mid_{H_\chi}\right\rangle", no head, bend left] \\
{} \arrow[d, "t", no head] \arrow[rru, no head] &  &  \\
H &  & 
\end{tikzcd}
\]


Choose an element $\sigma \in \Gal(H_\infty'/H)$ such that $\chi(\sigma) = \frac{1}{[H_\chi:H]}+\Ze$. We set $t := [H_\chi \cap H' : H]$ and fix
$\tau \in  \Gal(H_\infty'/H)$ such that $\tau|_{H_\chi} = 1$ and  $\tau|_{H'} = \varphi^t$. Recall that $d = [H' : H]$. It is  
easy to see that the set $\{ \sigma^i \tau^j | 0 \le i <t, 0 \le j < d/t \}$ constitutes a set of representatives of
$\Gal(H_\infty'/H)$ modulo  $\Gal(H_\infty'/H')$. For all $n \ge 0$ we define $\w_n$ as in (\ref{w def}) with respect to this set of representatives.

Then $\chi(\sigma^t) = \frac{1}{[H_\chi':H']}+\Ze$. Let $n$ be large enough so that $H_\chi' \subseteq H_n'$.
By Lemma \ref{lemma_propertiesvandw} b) and Hilbert's Theorem 90 there exists
an element $\beta \in (H_\chi')^\times$ such that $\beta^{\sigma^t-1} = \N_{H_n'/H_\chi'}(\w_n)$. By the definition of the pairing we derive
\begin{equation}\label{nice lemma eq 1}
(\w, \chi)_{H_\infty'/H'} = v_{H'}(\beta).
\end{equation}
If $\tilde\beta \in H_\chi^\times$ is such that  $\tilde\beta^{\sigma-1} = \N_{H_n'/H_\chi}(u_n)$, then 
\begin{equation}\label{nice lemma eq 2}
(u, \chi)_{H_\infty'/H} = v_{H}(\tilde\beta).
\end{equation}
Now we compute
\begin{eqnarray*}
  \tilde\beta^{\sigma^t-1} &=& \tilde\beta^{(\sigma-1)(1+\sigma + \ldots + \sigma^{t-1})} \\
                                        &=& \N_{H_n'/H_\chi}(u_n)^{(1+\sigma + \ldots + \sigma^{t-1})} \\
                                      &=& \left( \prod_{j=0}^{d/t-1}  \tau^j \left( \N_{H_n'/H_\chi'}(u_n)\right) \right)^{(1+\sigma + \ldots + \sigma^{t-1})} \\
                                       &=& \N_{H_n'/H_\chi'}\left( \prod_{j=0}^{{d}/{t}-1}  \prod_{i=0}^{t-1}\sigma^i\tau^j (u_n) \right) \\
                                       &=&  \N_{H_n'/H_\chi'}\left( \w_n \right).
\end{eqnarray*}
It follows that $\beta/\tilde\beta \in (H')^\times$ and hence $v_{H'}(\beta) \equiv v_{H'}(\tilde\beta) \pmod{\Ze}$. Since $H'/H$ is unramified
we obtain furthermore $v_{H'}(\tilde\beta) = v_{H}(\tilde\beta)$, which combined with (\ref{nice lemma eq 1}) and (\ref{nice lemma eq 2})
leads to  $(\w, \chi)_{H_\infty'/H'} =  (u, \chi)_{H_\infty'/H}$. 
The result  is now immediate from Theorem \ref{Seiriki's theorem} 
applied for the norm-coherent sequence $\w$ together with Lemma \ref{lemma_propertiesvandw}, part c).
\end{proof}

\section{Proof of the Main Theorem}

We recall that $F=k(\frf)$ with an ideal $\frf$ such that $\frf_L \mid \frf,\; w(\frf) = 1$ and $\frp \nmid \frf$.
By \cite[Ch.~II, Lemma~1.4]{dS87} there exists an elliptic curve $E$ defined over $F$ with complex multiplication by $\Ok$
and such that $F(E_{tor})$ is an abelian extension of $k$. The associated Gr\"o\ss encharacter is of the form 
$\psi_{E/F} = \varphi \circ \N_{F/k}$ with a {Gr\"o\ss encharacter} $\varphi$ of $k$ of
infinity type $(1,0)$ and conductor $\frf$. 
Note that $E$ has good reduction at all primes of $F$ above $\frp$.

Let $\frP$ be a prime of $F$ above $\frp$. Let $\iota: \Q^{c} \hookrightarrow \Q_{p}^{c}$ be a field embedding defining $\frP$. 
Via $\iota$ we view elements of $\Q^{c}$ as elements of $\Q_{p}^{c}$, but we sometimes omit $\iota$ in our notation. 
Furthermore, for any finite extension $M/k$ we write $\widetilde{M}$ for the completion of $\iota(M)$.

Since $E$ has good reduction at $\frP$, we may and will fix a Weierstrass model over the localization $\O_{F, \frP}$ of $\O_{F}$ at $\frP$ 
such that the associated discriminant $\Delta_{E}$ is a unit in $\O_{F, \frP}$. Replacing $E$ by one of its conjugates, if necessary, 
we may assume that the period lattice associated with the standard invariant differential of our 
fixed Weierstrass model is given by $\Omega\frf$ with $\Omega \in \Ce^\times$.

Let $\hat{E}$ be the one-parameter formal group law of $E$ with respect to the parameter \hbox{$t= -2x/y$}. 
Then $\hat{E}$ is defined over the completion $\O_{\tilde{F}} $ of $\O_{F, \frP}$.
By \cite[Ch.~II, Lemma~1.10]{dS87} this formal group $\hat{E}$ is a relative Lubin-Tate group of height two (since we always assume that $p$ is non-split) 
with respect to the unramified extension $\widetilde{F}/\widetilde{k}$.

For any integral ideal $\frc$ of $k$ we write $E[\frc]$ for the subgroup of $E(\Qu^c)$ annihilated by all elements $\alpha \in \frc$.
From \cite[Ch.~II, Prop.~1.6, Prop.~1.9 (i)]{dS87} we deduce that $F(E[\frp^{n}]) = k(\frf\frp^{n})$ (exponent changed) for all $n \ge 0$.

We set $H := \widetilde{k} = \widetilde{L}$, $H' := \widetilde{F}$ and resume the notation of Section~\ref{Rel Lubin-Tate}. 
In particular, $\hat{E}$ is a Lubin-Tate formal group relative to the unramified extension $H'/H$.
Similarly as in Section~\ref{Rel Lubin-Tate} we let $W^n(\hat{E})$ denote the group of division points of level $n$ in $\hat{E}$. Then
\cite[Ch.~II, Prop.~1.6, Prop.~1.8]{dS87} implies that \hbox{$\widetilde{k(\frf\frp^{n})}= H'(W^{n}(\hat{E})) = H_{n}'$} for $n \ge 0$.

We set $u_{n}:= \iota(\psi(1, \frf\frp^{n+1}, \fra))$ for $n \ge 0$  and get a 
norm-coherent sequence \hbox{$u:= \left( u_{n} \right)_{n=0}^{\infty} \in \varprojlim_{n} \O^{\times}_{H_n'}$} 
with an associated Coleman power series \hbox{$\Col_{u} \in \O_{\locKp}[[T]]$}
depending on a choice of a generator $\omega = (\omega_n)_{n \ge 0}$ of the Tate module of $\hat{E}$. 

As explained in \cite[Ch.~II.4.4]{dS87} or \cite[Sec.~4.3]{Bl04} one can choose a generator of the Tate module of $\hat{E}$ such that
Proposition \ref{prop_Colemanpslocaltoglobal} below holds. { Note that in \cite[Ch.~II.4]{dS87} it is assumed that $p$ is split in $k/\Qu$. However, one can show that this assumption is not needed for proving the following result.}

\begin{proposition}
\label{prop_Colemanpslocaltoglobal}
Let $P(z) \in \C[[z]]$ be the Taylor series expansion of $\psi({\Omega} - z; \Omega\frf , \fra)$ at $z=0$. 
Let $\lambda_{\hat{E}}$ denote the formal logarithm associated with $\hat{E}$ normalized such that $\lambda_{\hat{E}}(0) = 1$.
Then $P(z) \in F[[z]] \subseteq \locKp[[z]]$, and moreover:
\begin{enumerate}[label=\alph*)]
\item $P(\lambda_{\hat{E}}(T)) \in \O_{\locKp}[[T]].$
\item $\Col_{u}(T)=P(\lambda_{\hat{E}}(T)) $.
\item $\Col_{u}(0)= \iota(\psi(1; \frf , \fra))$.
\end{enumerate}
\end{proposition}
In order to prove Theorem \ref{thm_main} we will fix $n \ge 0$ and show the congruence of Remark \ref{rem_equivalencemainconj 2}.
We first consider the special case where $L$ is the full decomposition field of $\frp$ in $F/k$.
Let $e$ denote the ramification degree of $H/\Qp$. We set 
\[
\tilde{k} = \tilde{L} = H, \quad \tilde{F} = H', \quad H_\infty' = \widetilde{K_\infty} \text{ and }
\widetilde{K_n} =  H_{s+en}' 
\] 
and obtain a diagram of field extensions of $\tilde{k} = \tilde{L}$ of the same shape as the diagram  after Definition \ref{psi_n def}.

By a slight abuse of notation we write $\Gamma$ (resp. $\Delta$) for the Galois group of $H_\infty'/H_s'$ (resp. $H_\infty' / \widetilde{L_\infty}$).
For $i \in \{1,2\}$ we define a character $\chi_{i,n} \colon \Gal(H_\infty'/H) \lra \Qu/\Ze$ by 
\begin{eqnarray*}
  \chi_{i,n}&\colon& \Gal(H_\infty'/H) = \Delta \times \Gamma \twoheadrightarrow \Gamma \stackrel{\chiell}\lra 1+\frp_H^s \twoheadrightarrow \\ 
&&  \frac{1+\frp_H^s}{1+\frp_H^{s+en}} \stackrel{\log_p}\lra \frac{\frp_H^s}{\frp_H^{s+en}}  \stackrel{\pi_{\omega_j}}\lra \Zp/p^n\Zp \simeq \frac{1}{p^n}\Ze/\Ze.
\end{eqnarray*}

By construction $\ker(\chi_{i,n}) = \Gal(H_\infty'/\widetilde{L_{i,n}})$ and $\chi_{i,n}(\gamma_j) = \frac{1}{[\widetilde{L_{i,n}} : H]} + \Ze = 
 \frac{1}{p^n} + \Ze$.

\begin{remark}\label{local_global_reciprocity}
We emphasize that $\chiell$ arises as the inverse of the projective limit of global Artin isomorphisms  
$\frac{1+\frf\frp^s}{1+\frf\frp^{s+n}} \lra \Gal(k(\frf\frp^{s+n})/k(\frf\frp^{s}))$. The computations in case ($\gamma$) 
of \cite[Ch.~II.4.4.3]{Gras} show that the composite
\[
1+\frp^s\O_H \stackrel{\rec_H}\lra \Gamma \stackrel{\chiell}\lra 1 + \frp^s\O_H
\]
is given by $\alpha \mapsto \alpha^{-1}$.
\end{remark}

By Corollary \ref{cor_sei} we obtain the following equality in $\Qu/\Ze$
\[
(u, \chi_{i,n})_{H_\infty'/H} = - \chi_{i,n}(\rec_H(\N_{H'/H}(\Col_u(0)))).
\]
We first compute the right hand side of this equality. 
By construction of $\chi_{i,n}$, Proposition \ref{prop_Colemanpslocaltoglobal}
and Remark \ref{local_global_reciprocity} we get
\begin{eqnarray}
- \chi_{i,n}(\rec_H(\N_{H'/H}(\Col_u(0)))) &=&  \frac{1}{p^n}\pi_{\omega_j} \left( \log_p(\N_{H'/H}(\Col_u(0))) \right)  \nonumber\\
                                                             &=&  \frac{1}{p^n}\pi_{\omega_j} \left( \log_p(\iota(\N_{F/L}(\psi(1;\frf,\fra))) \right) \label{last but one eq}.
\end{eqnarray}
For the computation of the left hand side we note that $\iota(\beta_{i,n})^{\gamma_j - 1} = \iota(\epsilon_{i,n}) = \N_{{\widetilde{K_{n}}}/\widetilde{L_{i,n}}}(u_{s+n})$,
so that by definition of Seiriki's pairing we obtain
\begin{equation}\label{last eq}
(u, \chi_{i,n})_{H_\infty'/H} = \frac{1}{p^n} v_{\widetilde{L_{i,n}}}(\iota(\beta_{i,n})).
\end{equation}
Combining (\ref{last but one eq}) and (\ref{last eq}) we derive the congruence of Remark \ref{rem_equivalencemainconj 2}. 
This concludes the proof of Theorem \ref{thm_main} in the case that $L$ is the full decomposition field of $\frp$ in $F/k$.

The general case follows from the special case precisely in the same way as in \cite[Sec.~4.3]{Bl04}.

\bibliographystyle{abbrv}

\footnotesize

  Werner Bley, \textsc{Department of Mathematics, Ludwig-Maximilians-Universit\"at M\"unchen}\par\nopagebreak
  \textit{E-mail address: } \texttt{bley@math.lmu.de}

  \medskip

  Martin Hofer, \textsc{Department of Mathematics, Ludwig-Maximilians-Universit\"at M\"unchen}\par\nopagebreak
  \textit{E-mail address: }\texttt{hofer@math.lmu.de}

\end{document}